\documentclass[12pt]{article}

\usepackage{amsmath, amssymb, amsthm, amsfonts}
\usepackage{hyperref}
\usepackage{bm}
\usepackage{comment}
\usepackage{xcolor}
\hypersetup{
  hidelinks,
  pdftitle={Spherical Cap L2 Discrepancy -- Blessing of Dimensionality and a Distance-Kernel Variant},
  pdfauthor={Johann S. Brauchart, Josef Dick, Friedrich Pillichshammer}
}

\newcommand{\cp}{{\rm cap}}

\newtheorem{theorem}{Theorem}
\newtheorem{lemma}[theorem]{Lemma}
\newtheorem{corollary}[theorem]{Corollary}
\newtheorem{proposition}[theorem]{Proposition}

\newtheorem{definition}[theorem]{Definition}

\newtheorem{remark}[theorem]{Remark}

\title{Spherical Cap $L_2$ Discrepancy -- Blessing of Dimensionality and a Distance-Kernel Variant}
\author{Johann S. Brauchart\thanks{Institut für Analysis und Zahlentheorie, TU Graz, Kopernikusgasse 24/II, 8010 Graz, Austria (\url{j.brauchart@tugraz.at}). The research of J.~S.~B. was funded in whole or in part by the OeAD (Austria's Agency for Education and Internationalisation) Project No. BG 07/2025. }, Josef Dick\thanks{School of Mathematics and Statistics, The University of New South Wales Sydney, 2052 NSW, Australia (\url{josef.dick@unsw.edu.au}). The work of J.~D.\ is supported by ARC grant DP220101811.}, Friedrich Pillichshammer\thanks{Institut für Finanzmathematik und Angewandte Zahlentheorie, Johannes Kepler Universität Linz, Altenbergerstra{\ss}e 69, 4040 Linz, Austria (\url{friedrich.pillichshammer@jku.at}). The research of F.~P. was funded in part by the Austrian Science Fund (FWF) 10.55776/PAT1173425. For open access purposes, the authors have applied a CC BY public copyright license to any author-accepted manuscript version arising from this submission.}}
\date{}

\begin{document}

\maketitle

\begin{abstract}
We prove a dimension-explicit upper bound for the information-complexity (also called the inverse of the discrepancy) of the classical spherical cap $L_2$ discrepancy on the $d$-dimensional sphere $\mathbb{S}^d$.  For every fixed $\varepsilon\in(0,1)$, this bound implies that a single point suffices in all sufficiently large dimensions, which shows a ``blessing of dimensionality'' for the associated numerical integration problem.  We then introduce a one-parameter family of distance reproducing kernels for the Sobolev space $H^{(d+1)/2}(\mathbb{S}^d)$.  For the balanced kernel $K(\bm{x},\bm{y})=1-\tfrac{1}{\sqrt{2}}\|\bm{x}-\bm{y}\|$, a lower bound together with a random-point upper bound shows that the normalized information-complexity is $\Theta_\varepsilon(d)$ for every fixed $\varepsilon\in(0,1)$.  Thus its dependence on $d$ is determined up to constants depending only on $\varepsilon$. 
Finally, we study the one-sided spherical cap discrepancy obtained by restricting the height parameter to $[0,1]$.  We derive its kernel and Stolarsky's invariance principle explicitly.  A sharp mean-distance inequality, attained by regular simplices, determines its normalized information-complexity exactly in all sufficiently large dimensions and proves strong polynomial tractability with optimal $\varepsilon$-exponent two.
\end{abstract}

\centerline{\begin{minipage}[hc]{130mm}{
}
\end{minipage}}

\section{Introduction}

Discrepancy on the sphere quantifies how well a finite point set reproduces the uniform surface measure when tested against indicator functions of spherical caps. Among the various notions of spherical discrepancy, the \emph{spherical cap $L_2$ discrepancy} occupies a central position: it is rotationally invariant, admits a reproducing-kernel representation, and is tied -- via Stolarsky's invariance principle -- to geometric energies on the sphere \cite{BD13a,BD13b}. In numerical integration on $\mathbb{S}^d$, this discrepancy governs worst-case errors for an associated reproducing kernel Hilbert space (RKHS) of functions, thus providing an information-complexity lens on quadrature over $\mathbb{S}^d$.

For the \emph{classical} spherical cap $L_2$ discrepancy on $\mathbb{S}^d$, we prove a decreasing dimension-explicit upper bound for the inverse discrepancy (equivalently, the information-complexity).  In particular, for every fixed $\varepsilon\in(0,1)$, a single point suffices once $d$ is large enough; we call this the \emph{blessing of dimensionality}.  This contrasts with the familiar ``curse of dimensionality'' in high-dimensional numerical analysis.  We then study a dimension-dependent family of distance reproducing kernels on $\mathbb{S}^d$.  For a balanced member of this family, the normalized information-complexity has exact order $d$ for every fixed error tolerance.  The lower bound closes the gap between the order-$\sqrt d$ dependence supplied by the general discrepancy estimate and the order-$d$ random-point upper bound.

We also consider a nonsymmetric, one-sided restriction in which only cap heights $t\in[0,1]$ are used. Its kernel is a constant multiple of $1-\tfrac12\|\bm{x}-\bm{y}\|$. Although the squared one-sided discrepancy of every nonempty equal-weight rule is exactly one half of the classical squared cap discrepancy, its initial error has a different scale. A sharp upper bound for the mean pairwise  distance, with equality for regular simplices, gives the exact normalized information-complexity once $d$ is sufficiently large.  Its eventual value is asymptotic to $\tfrac{1+\sqrt2}{2}\varepsilon^{-2}$ as $\varepsilon\downarrow0$.  In particular, the problem is strongly polynomially tractable uniformly in $d$, with optimal $\varepsilon$-exponent two.

\paragraph{Setting and notation.}
Let $\mathbb{S}^d:=\{ \bm{x} \in\mathbb{R}^{d+1}:\| \bm{x} \|_2=1\}$ be equipped with the normalized surface area measure $\sigma_d$, so $\sigma_d(\mathbb{S}^d)=1$, where $\| \bm{\cdot} \|_2 := \sqrt{\langle \bm{\cdot} , \bm{\cdot} \rangle}$ is the norm induced by the usual inner product in $\mathbb{R}^{d+1}$. 
The chordal distance between two points on $\mathbb{S}^d$ satisfies 
\begin{equation*}
\left\| \bm{x} - \bm{y} \right\|^2
= 
2 \left( 1 - \langle \bm{x} , \bm{y} \rangle \right), 
\qquad 
\bm{x} , \bm{y} \in \mathbb{S}^d.
\end{equation*}
For $\bm{z} \in\mathbb{S}^d$ and $t\in[-1,1]$ let 
\[
C(\bm{z}; t) \;:=\;\{\bm{y} \in\mathbb{S}^d:\langle \bm{z}, \bm{y} \rangle\ge t\}
\]
denote the spherical cap with center $\bm{z}$ and height $t$.

We study the distribution properties of $N$-point subsets of
$\mathbb{S}^d$, denoted by
\[
P_{N,d}:=\{\bm{x}_1,\ldots,\bm{x}_N\}\subseteq\mathbb{S}^d.
\]
The subscripts specify the cardinality and the dimension.  A central measure of
uniformity is the spherical cap $L_2$ discrepancy.  We write
$\bm{1}_{C(\bm{x};t)}$ for the indicator function of $C(\bm{x};t)$.

\begin{definition}\label{def1}\rm
For $d \in \mathbb{N}$ the spherical cap $L_2$ discrepancy of an $N$-point set $P_{N,d} = \{ \bm{x}_1, \ldots, \bm{x}_N\} \subseteq \mathbb{S}^d$ is given by
\begin{equation*}
L_2^{\cp}(P_{N,d}):= \left( \int_{-1}^1 \int_{\mathbb{S}^d} \left( \frac{1}{N} \sum_{n=1}^N \bm{1}_{C(\bm{z}; t)}(\bm{x}_n) - \sigma_d(C(\bm{z}; t)) \right)^2 \,\mathrm{d} \sigma_d(\bm{z}) \,\mathrm{d} t \right)^{1/2}.
\end{equation*}
\end{definition}
As a reference value, we define the initial spherical cap $L_2$ discrepancy by
\begin{equation*}
L_2^{\cp}(\emptyset) := \left(\int_{-1}^1 \int_{\mathbb{S}^d} \left( \sigma_d(C(\bm{z}; t)) \right)^2 \,\mathrm{d} \sigma_d(\bm{z}) \,\mathrm{d} t\right)^{1/2}.
\end{equation*}

In addition to its geometric interpretation, the spherical cap $L_2$ discrepancy is also intimately related to the worst-case error of quasi-Monte Carlo rules for the numerical integration of functions over $\mathbb{S}^d$ with respect to $\sigma_d$, see \cite{BD13a,BSSW14,DP14a}.

\paragraph{Information-complexity.} In this paper, we are interested in information-based complexity, where one typically examines the following inverse problem: \emph{determine how much information is required to reduce the initial discrepancy to an $\varepsilon$-fraction of its original value for $\varepsilon \in (0,1)$.} The corresponding measure is the information-complexity, also known in the discrepancy setting as the inverse discrepancy.

\begin{definition}\rm
For $\varepsilon \in (0,1)$ and $d \in \mathbb{N}$, the inverse of the spherical cap $L_2$ discrepancy is defined as 
\begin{equation*}
N^{\cp}(\varepsilon, d) := \inf \{ N \in \mathbb{N}: \exists P_{N,d} \subseteq \mathbb{S}^d \ \text{such that}\ L_2^{\cp}(P_{N,d}) \le \varepsilon L_2^{\cp}(\emptyset) \}.
\end{equation*}
\end{definition}
Equivalently, $N^{\cp}(\varepsilon,d)$ is the minimal number of function evaluations needed by an equal-weight quadrature rule to guarantee normalized worst-case error at most $\varepsilon$ in the unit ball of the corresponding function class. We stress that our upper and lower bounds are dimension-explicit, enabling a direct comparison of $N^{\cp}(\varepsilon,d)$ across $d$.

\paragraph{Main contributions.} The main contributions of the present work are as follows:
\begin{itemize}
  \item \textbf{Blessing of dimensionality for the classical spherical cap $L_2$ discrepancy.}  We prove a decreasing upper bound for $N^{\cp}(\varepsilon,d)$.  For every fixed $\varepsilon\in(0,1)$, the bound equals one in all sufficiently large dimensions, and therefore $N^{\cp}(\varepsilon,d)=1$ there.  This eventual one-point solvability is the precise sense in which the associated worst-case integration problem exhibits a \emph{blessing of dimensionality}.

  This computational ease also raises the question whether the
  classical normalization remains informative in high dimension.  The second
  part of the paper addresses that issue.
 
  \item \textbf{A distance-kernel discrepancy with stabilized difficulty.} We introduce a family of reproducing kernels whose nonconstant distance part can be balanced against the dimension. For the kernel $1-\tfrac{1}{\sqrt{2}}\|\bm{x}-\bm{y}\|$, Theorem~\ref{thm1} and Corollary~\ref{cor:balanced-linear} establish the exact dimension dependence $\Theta_\varepsilon(d)$ of the normalized information-complexity.

  \item \textbf{A one-sided spherical cap discrepancy.} By restricting the cap-height parameter to $[0,1]$, we obtain the kernel $C_d(1-\tfrac12\|\bm{x}-\bm{y}\|)$, with a quantity $C_d$ introduced in \eqref{def:Cd} that only depends on $d$, and a corresponding Stolarsky invariance principle.  Theorem~\ref{thm:one-sided-tractability} and Corollary~\ref{cor:one-sided-eventual} give an explicit exact formula for its normalized information-complexity in all sufficiently large dimensions.  The formula closes the dimension gap and shows that its eventual value is $\tfrac{1+\sqrt2}{2}\varepsilon^{-2}+\mathcal O(1)$ as $\varepsilon\downarrow0$; consequently, the optimal strong-tractability exponent is two.
\end{itemize}

Besides the spherical cap $L_2$ discrepancy, numerous other figures of
merit quantify the quality of point distributions on the sphere.  The
dimension dependence of one such criterion, a type of diaphony, was studied in
\cite[Section~3.4]{ChS13}.  Related Stolarsky-type principles on projective
spaces were developed in \cite{Skr2020}.

\paragraph{Why the blessing of dimensionality arises.}
The underlying phenomenon is concentration of measure.  For fixed
$\bm{x}\in\mathbb{S}^d$ and a uniformly distributed point $Y$, the inner
product $\langle\bm{x},Y\rangle$ concentrates near zero as $d$ grows.  More
precisely, for fixed $t\in(0,1)$ (see Appendix~\ref{sec_dist_int}),
\begin{equation}\label{asym:Cxt}
\sigma_d(C(\bm{x}; t))
\sim \frac{(1-t^2)^{\tfrac{d}{2}}}{t\sqrt{2\pi d}}
=o\!\left(\frac{\mathrm{e}^{-\tfrac{d}{2} t^2}}{\sqrt d}\right)
\qquad \text{as $d \to \infty$.}
\end{equation}
Thus the measure of a cap with any fixed $t\in(0,1)$ decreases exponentially with $d$; by complementarity, a cap with fixed $t\in(-1,0)$ has measure tending exponentially to one.  At $t=0$ the cap is a hemisphere and has measure $1/2$, while the transition between the two regimes occurs on the scale $|t|=\mathcal O(1/\sqrt{d})$.  Consequently, for most heights in the defining integral, both the empirical and uniform cap masses are close to zero or one, and only a shrinking neighborhood of the hemispherical level remains nontrivial.  This explains the degeneration of the classical discrepancy and motivates modifying the RKHS normalization so that its nonconstant component remains visible in high dimension.

\paragraph{Stolarsky's invariance principle and numerical integration.}

In \cite{BD13a}, the first two authors considered the Sobolev space
$H^{(d+1)/2}(\mathbb{S}^d)$, endowed with the reproducing kernel
\begin{equation}\label{def:ker}
K_{1,d}(\bm{x}, \bm{y}) := 1 - C_d \|\bm{x} - \bm{y}\|, \qquad \bm{x}, \bm{y} \in \mathbb{S}^d, 
\end{equation}
where (see \cite[Eq. (2.4)]{BD13a})
\begin{equation}\label{def:Cd}
C_d := \frac{1}{d} \frac{\Gamma(\tfrac{d+1}{2})}{\sqrt{\pi} \ \Gamma(\tfrac{d}{2})} \sim \frac{1}{\sqrt{2 \pi d}} \qquad \mbox{as } d \to \infty.
\end{equation}
Here $\Gamma(s):=\int_0^\infty t^{s-1}{\rm e}^{-t}\,\mathrm{d}t$
for $\Re s>0$ is the gamma function. They showed that the spherical cap $L_2$ discrepancy of
$P_{N,d}$ equals the worst-case integration error of the equal-weight rule
supported on $P_{N,d}$.
Furthermore, the squared spherical cap $L_2$ discrepancy can be computed either using Stolarsky's invariance principle 
\begin{equation}
\label{eq:Stolarsky.principle}
(L_2^{\cp}(P_{N,d}))^2 = C_d  \int_{\mathbb{S}^d} \int_{\mathbb{S}^d} \|\bm{x} - \bm{y}\| \,\mathrm{d} \sigma_d(\bm{x}) \,\mathrm{d}\sigma_d(\bm{y}) - C_d \frac{1}{N^2} \sum_{m,n=1}^N \|\bm{x}_m -\bm{x}_n\|    
\end{equation}
or the well-known reproducing kernel Hilbert space representation of the squared worst-case error
\begin{align} \label{L2_kernel}
\begin{split}
(L_2^{\cp}(P_{N,d}))^2
&=
\int_{\mathbb{S}^d} \int_{\mathbb{S}^d} K_{1,d}(\bm{x}, \bm{y}) \,\mathrm{d} \sigma_d(\bm{x}) \,\mathrm{d}\sigma_d(\bm{y}) 
-
\frac{2}{N} \sum_{n=1}^N \int_{\mathbb{S}^d} K_{1,d}(\bm{x}, \bm{x}_n) \,\mathrm{d} \sigma_d(\bm{x}) 
\\
&\phantom{=}+
\frac{1}{N^2} \sum_{m,n=1}^N K_{1,d}(\bm{x}_m, \bm{x}_n).
\end{split} 
\\
&=
\frac{1}{N^2} \sum_{m,n=1}^N K_{1,d}(\bm{x}_m, \bm{x}_n) - \int_{\mathbb{S}^d} \int_{\mathbb{S}^d} K_{1,d}(\bm{x}, \bm{y}) \,\mathrm{d} \sigma_d(\bm{x}) \,\mathrm{d}\sigma_d(\bm{y}). \notag
\end{align}
The kernel representation gives the initial spherical cap $L_2$
discrepancy directly:
\begin{equation}\label{init:err:fo}
(L_2^{\cp}(\emptyset))^2 = \int_{\mathbb{S}^d} \int_{\mathbb{S}^d} K_{1,d}(\bm{x}, \bm{y}) \,\mathrm{d} \sigma_d(\bm{x}) \,\mathrm{d} \sigma_d(\bm{y}) = 1 - C_d I_d,
\end{equation}
where $I_d$ denotes the distance integral given as
\begin{equation}\label{def:Id}
I_d :=  \int_{\mathbb{S}^d} \int_{\mathbb{S}^d} \|\bm{x} - \bm{y}\| \,\mathrm{d} \sigma_d(\bm{x}) \,\mathrm{d} \sigma_d(\bm{y}).
\end{equation}

We collect the estimates for $I_d$ and $C_d$ that will be used below.
Their proof is deferred to Appendix~\ref{sec_dist_int}.
\begin{lemma}
\label{lem:C.d.I.d.inequalities}
For every $d\in\mathbb{N}$, the quantities $I_d$ from \eqref{def:Id}
and $C_d$ from \eqref{def:Cd} satisfy:
\begin{enumerate}
\item $\displaystyle \sqrt{2} \, \sqrt{1 - \tfrac{1}{d}} < I_d < \sqrt{2}$; furthermore, $\displaystyle I_d = \sqrt{2} \left( 1 - \tfrac{1}{8d} + \mathcal{O}\big( \tfrac{1}{d^2} \big) \right)$ as $d \to \infty$;
\item $\displaystyle \frac{1}{\sqrt{2 \pi \, d}} \, \sqrt{1 - \tfrac{1}{d}} < C_d < \frac{1}{\sqrt{2 \pi \, d}}$;
\item $\displaystyle \frac{1}{\sqrt{\pi \, d}} \left( 1 - \tfrac{1}{d} \right) < C_d \, I_d < \frac{1}{\sqrt{\pi \, d}}$;
\item $\displaystyle  \frac{1 - \frac{1}{d}}{\sqrt{\pi \, d} - \left( 1 - \frac{1}{d} \right)} < \frac{C_d \, I_d}{1 - C_d \, I_d} < \frac{1}{\sqrt{\pi \, d} - 1}$.
\end{enumerate}
\end{lemma}

The leading-order asymptotics of $I_d$, $C_d$, $C_dI_d$, and
$C_dI_d/(1-C_dI_d)$ as $d\to\infty$ follow from Lemma~\ref{lem:C.d.I.d.inequalities}.

For the reproducing kernel \eqref{def:ker} for $\bm{x},\bm{y} \in \mathbb{S}^d$ we obviously have
\begin{equation}\label{kernel_bounds}
1 - 2 C_d \le K_{1,d}(\bm{x}, \bm{y}) \le 1. 
\end{equation}
Thus, applying Lemma~\ref{lem:C.d.I.d.inequalities} yields $\sup_{\bm{x},\bm{y} \in \mathbb{S}^d} |K_{1,d}(\bm{x}, \bm{y}) -1| \le 2 C_d \to 0$ as $d \to \infty$.

The distance integral has the explicit expression (see, e.g., \cite{BSSW14})
\begin{equation}
\label{eq:I.d.value}
I_d = 2^d \, \frac{\Gamma( \tfrac{d + 1}{2} )^2}{\sqrt{\pi} \, \Gamma( d + \tfrac{1}{2} )}
\sim
\sqrt{2} 
\qquad
\text{as $d \to \infty$,}
\end{equation}
where the asymptotic formula is obtained from Lemma~\ref{lem:C.d.I.d.inequalities}. Since, again by Lemma~\ref{lem:C.d.I.d.inequalities}, $C_d \, I_d \sim 1/\sqrt{\pi d}$ as $d \to \infty$, we have the more informative expansion
\begin{equation*}
L_2^{\cp}(\emptyset)
=\sqrt{1-C_dI_d}
=1-\frac{1}{2\sqrt{\pi d}}+o\!\left(\frac{1}{\sqrt{d}}\right)
\qquad\text{as $d\to\infty$.}
\end{equation*}

The blessing of dimensionality is also reflected in the reproducing kernels $K_{1,d}$, which converge pointwise and uniformly to the constant kernel $1$.  Although the underlying spaces live on different spheres and therefore do not literally converge as a fixed sequence of function spaces, their normalized nonconstant kernel component vanishes.  Constant functions are integrated exactly by any one-point rule, which is consistent with the eventual one-point conclusion of Theorem~\ref{thm:1}.

\paragraph{The distance-kernel discrepancy.}
For $\alpha_d > C_dI_d$ (we call such $\alpha_d$ \emph{admissible}), we will consider the distance reproducing kernel
\[
K_{\alpha_d,d}(\bm{x},\bm{y})
:=
1-\frac{C_d}{\alpha_d}\|\bm{x}-\bm{y}\|, \qquad \bm{x}, \bm{y} \in \mathbb{S}^d.
\]
The parameter $\alpha_d$ changes the equivalent norm on
$H^{(d+1)/2}(\mathbb{S}^d)$ and thereby changes the normalized information
complexity.  The balanced choice $\alpha_d=\sqrt2\,C_d$ makes the chordal
coefficient independent of $d$. In Section~\ref{sec:tract} we will see that in this case the reproducing kernel 
\begin{equation}\label{Ksqrt2Cd}
K_{\sqrt{2} C_d,d}(\bm{x}, \bm{y}) = 1 - \frac{1}{\sqrt{2}} \|\bm{x} - \bm{y}\|, \qquad \bm{x},\bm{y} \in \mathbb{S}^d
\end{equation}
has a nonconstant part that is independent of $d$ and therefore does not
degenerate to the constant kernel $1$.  Corollary~\ref{cor:balanced-linear}
shows that, for every fixed
$\varepsilon\in(0,1)$, the normalized information-complexity has the exact
order $d$.

We emphasize already here that all admissible distance kernels have the same optimal nonempty equal-weight point sets. Those point sets maximize the sum of mutual distances by Stolarsky's invariance principle~\eqref{eq:Stolarsky.principle}. The different normalized complexities arise from their different initial errors, which are given in \eqref{init:err:dist:ker}.

The genuinely cap-based one-sided restriction $t\in[0,1]$ is treated in
Section~\ref{sec:one-sided}, and the hemisphere discrepancy is recalled in
Section~\ref{sec_hemisphere}.

\paragraph{Comparison with the discrepancy on the unit cube.}

The canonical example is the $L_2$ discrepancy of $N$-point sets in the $d$-dimensional unit cube $[0,1)^d$ with respect to the uniform (Lebesgue) measure $\lambda$ given by
\begin{equation}\label{def:L2discCube}
L_2(P_{N,d})=\left(\int_{[0,1]^d} \left( \frac{1}{N} \sum_{n=1}^N \bm{1}_{[\bm{0}, \bm{t})}(\bm{x}_n) - \lambda([\bm{0}, \bm{t})) \right)^2 {\rm d} \bm{t}\right)^{1/2},
\end{equation}
where $P_{N,d}=\{\bm{x}_1,\ldots,\bm{x}_N\} \subseteq [0,1)^d$, $[\bm{0},\bm{t})$ denotes the axis-parallel box anchored at the origin with opposite corner $\bm{t}\in[0,1]^d$, and $\bm{1}_{[\bm{0},\bm{t})}$ is the indicator function of the box $[\bm{0},\bm{t})$. As on the sphere, the $L_2$ discrepancy on the unit cube coincides with the worst-case error of quasi-Monte Carlo (QMC) rules for numerical integration. For background, see \cite[Chapter~9]{NW10}.

For the $L_2$ discrepancy on the unit cube \eqref{def:L2discCube} it is known that the number of points required to reduce the initial discrepancy by a factor of $\varepsilon$ increases exponentially with the dimension -- a behavior termed ``curse of dimensionality'', see, e.g., \cite{NP25b,NW10,Wo99}. 

Various other notions of how the information-complexity $N(\varepsilon,d)$ depends on $d$ have been discussed in the literature. If $N(\varepsilon, d)$ grows at most polynomially in $d$ and $\varepsilon^{-1}$, the problem is \emph{polynomially tractable}; if the growth is at most polynomial in $\varepsilon^{-1}$ uniformly in $d$, one speaks of \emph{strong polynomial tractability}. Recent years have seen many discrepancy notions classified according to the growth of their inverse; see, e.g., \cite{NP25a,NW08,NW10}.

\paragraph{Interpretation of the results.}
The blessing property established here is specific to the normalized classical
spherical cap criterion; it should not be read as a claim that arbitrary
high-dimensional integration problems on spheres become easier.  The
balanced distance-kernel discrepancy deliberately prevents the nonconstant
kernel component from vanishing and has normalized information-complexity of
exact order $d$ for fixed accuracy.  The one-sided cap discrepancy provides a
different comparison: it retains all positive cap heights, and for every fixed
accuracy its normalized information-complexity stabilizes at an explicitly
known finite value as $d\to\infty$.

\section{The blessing of dimensionality of the spherical cap
\texorpdfstring{$L_2$}{L2}~discrepancy}\label{sec:scd1}

The following dimension-explicit upper bound establishes the blessing
of dimensionality for the classical spherical cap $L_2$ discrepancy.

\begin{theorem}
\label{thm:1}
For every $\varepsilon\in(0,1)$ and $d \in \mathbb{N}$,
\begin{equation*}
N^{\cp}(\varepsilon, d) \leq \left\lceil \frac{1}{\varepsilon^2} \, \frac{1}{\sqrt{\pi  d} - 1} \right \rceil.
\end{equation*}
\end{theorem}

\begin{proof}
Let the points of
$P_{N,d}=\{\bm{x}_1,\ldots,\bm{x}_N\}$ be independent and identically uniformly distributed on $\mathbb{S}^d$.  By Stolarsky's invariance principle
\eqref{eq:Stolarsky.principle} and because of \eqref{def:Id},
\begin{equation}\label{exp_rand}
\begin{split}
\mathbb{E} \left[ (L^{\cp}_2(P_{N,d}))^2 \right] 
&= 
C_d \left( I_d - \frac{1}{N^2} \sum_{m,n=1}^N \mathbb{E}\left[ \| \bm{x}_m - \bm{x}_n\| \right] \right) \\
&= 
C_d \left( I_d - \frac{N^2-N}{N^2} I_d \right) = \frac{C_d I_d}{N},
\end{split}
\end{equation}
where $C_d$ is defined in \eqref{def:Cd}.\footnote{The expected squared
worst-case error for a broad class of reproducing kernels was derived in
\cite[Section~7]{BSSW14}.} It follows from \eqref{exp_rand} that, for
every $d,N\in\mathbb{N}$, there exists a point set $P_{N,d}$ such that
\begin{equation*}
L_2^{\cp}(P_{N,d}) \le \sqrt{\frac{C_d I_d}{N}}.
\end{equation*}
Comparing with the initial spherical cap $L_2$ discrepancy \eqref{init:err:fo}, we find that $N^{\cp}(\varepsilon, d) \le N$ for any $N \in \mathbb{N}$ which satisfies
\begin{equation*}
\sqrt{\frac{C_d I_d}{N (1-C_d I_d)}} \le \varepsilon \quad \mbox{ or, equivalently,}\quad N \ge \frac{1}{\varepsilon^2} \frac{C_d I_d}{1-C_d I_d}.
\end{equation*}
Therefore
\begin{equation*}
N^{\cp}(\varepsilon, d) \le \left\lceil \frac{1}{\varepsilon^2} \frac{C_d I_d}{1-C_d I_d} \right\rceil.
\end{equation*}
The result follows by item 4 of Lemma~\ref{lem:C.d.I.d.inequalities}.
\end{proof}

Theorem~\ref{thm:1} gives a decreasing upper bound; it does not assert that the exact quantity $N^{\cp}(\varepsilon,d)$ is monotone for every $d$.  It does imply the stronger eventual statement
\[
N^{\cp}(\varepsilon,d)=1
\qquad\text{whenever}\qquad
d\geq\tfrac{1}{\pi}(1+\varepsilon^{-2})^2.
\]
This conclusion is consistent with \eqref{kernel_bounds}: the nonconstant part of $K_{1,d}$ vanishes uniformly as $d\to\infty$.  For $N=1$, rotational invariance makes the location of $\bm{x}_1$ irrelevant. Thus, for any $\bm{x}_1 \in \mathbb{S}^d$, as $d\to\infty$, we have
\begin{align*}
(L_2^{\cp}(\{\bm{x}_1\}))^2
&=C_dI_d\sim\frac1{\sqrt{\pi d}} \quad \mbox{ and } \quad 
\frac{(L_2^{\cp}(\{\bm{x}_1\}))^2}
{(L_2^{\cp}(\emptyset))^2}
=\frac{C_dI_d}{1-C_dI_d}\rightarrow 0.
\end{align*}

\section{Tractability of the distance-kernel discrepancy}\label{sec:tract}

\paragraph{The distance-kernel family and its geometric representation.} Let $d\in\mathbb{N}$ and $\alpha_d > C_dI_d$. As in the Introduction, we call such $\alpha_d$ \emph{admissible}. We define the \emph{distance-kernel} by
\begin{equation}\label{Kalpha}
K_{\alpha_d,d}(\bm{x},\bm{y})
:=1-\frac{C_d}{\alpha_d}\|\bm{x}-\bm{y}\|,
\qquad \bm{x},\bm{y}\in\mathbb{S}^d,
\end{equation}
and put
\begin{equation}\label{eq:kernel-constant-component}
g_{\alpha_d,d}:=1-\frac{C_dI_d}{\alpha_d}>0.
\end{equation}
The parameter $\alpha_d$ changes the relative weights of the constant
and mean-zero components of the kernel. It is not a restriction on
spherical-cap heights. Nevertheless, there is a certain geometric interpretation, which we will briefly explain.

For $t\in[-1,1]$ let
\begin{equation}\label{eq:centred-cap-features}
\Delta(\bm{x}; \bm{z}; t )
:=\bm{1}_{C(\bm{z};t)}(\bm{x})-\sigma_d(C(\bm{z};t)).
\end{equation}
Obviously, the cap measure $\sigma_d(C(\bm{z};t))$ is independent of the centre $\bm{z}$.

\begin{proposition}\label{prop:centred-cap}
For every $d\in\mathbb{N}$ and $\alpha_d>C_dI_d$,
\begin{align}
K_{\alpha_d,d}(\bm{x},\bm{y})
&=g_{\alpha_d,d}
+\frac1{\alpha_d}\int_{-1}^{1}\int_{\mathbb{S}^d}
\Delta(\bm{x}; \bm{z}; t ) \Delta(\bm{y}; \bm{z}; t )
 \,\mathrm{d}\sigma_d(\bm{z})\,\mathrm{d}t.
\label{eq:distance-kernel-centred-caps}
\end{align}
Equivalently, define
\begin{equation*}
\Phi_{\alpha_d,d}(\bm{x};\bm{z},t)
:=
\frac{1}{\sqrt{\alpha_d}} \left( \Delta(\bm{x}; \bm{z}; t ) + \sqrt{\frac{1}{2} \left( \alpha_d - C_d I_d \right)} \right).
\end{equation*}
Then
\begin{equation}\label{eq:distance-kernel-feature-integral}
K_{\alpha_d,d}(\bm{x},\bm{y})
=\int_{-1}^{1}\int_{\mathbb{S}^d}
\Phi_{\alpha_d,d}(\bm{x};\bm{z},t)
\Phi_{\alpha_d,d}(\bm{y};\bm{z},t)
\,\mathrm{d}\sigma_d(\bm{z})\,\mathrm{d}t.
\end{equation}
In particular, $K_{\alpha_d,d}$ is positive semidefinite.
\end{proposition}

\begin{proof}
For convenience, we set $c := \sqrt{( \alpha_d - C_d I_d )/2}$ and observe that by rotational symmetry 
\begin{equation*}
\int_{\mathbb{S}^d} \Delta(\bm{x}; \bm{z}; t ) \,\mathrm{d}\sigma_d(\bm{z}) 
= 0.
\end{equation*}
Furthermore, 
\begin{equation*}
\Phi_{\alpha_d,d}(\bm{x};\bm{z},t)
= 
\frac{1}{\sqrt{\alpha_d}} \Big( \Delta(\bm{x}; \bm{z}; t ) + c \Big).
\end{equation*}
Then we get for the inner integral over the sphere, 
\begin{equation*}
\begin{split}
&\int_{\mathbb{S}^d} \Phi_{\alpha_d,d}(\bm{x};\bm{z},t) \Phi_{\alpha_d,d}(\bm{y};\bm{z},t) \, \mathrm{d}\sigma_d(\bm{z}) \\\ 
&\phantom{equals}=
\frac{1}{\alpha_d} \left( \int_{\mathbb{S}^d} \Delta(\bm{x}; \bm{z}; t ) \Delta(\bm{y}; \bm{z}; t ) \, \mathrm{d}\sigma_d(\bm{z}) + c \cdot 0 + c \cdot 0 + c^2  \right).
\end{split}
\end{equation*}
We use (cf. \cite{BD13a})
\begin{equation}\label{eq:full-height-cap-kernel}
\int_{-1}^{1}\int_{\mathbb{S}^d}
\bm{1}_{C(\bm{z};t)}(\bm{x})
\bm{1}_{C(\bm{z};t)}(\bm{y})
\,\mathrm{d}\sigma_d(\bm{z})\,\mathrm{d}t
=1-C_d\|\bm{x}-\bm{y}\|,
\end{equation}
and by integrating \eqref{eq:full-height-cap-kernel} over $\bm{x}, \bm{y} \in \mathbb{S}^d$ with respect to the uniform measure $\sigma_d$, we obtain
$\int_{-1}^{1}\sigma_d(C(\bm{z};t))^2\,\mathrm{d}t=1 - C_dI_d$. Therefore we get 
\begin{equation*}
\begin{split}
&\int_{-1}^{1}\int_{\mathbb{S}^d} \Phi_{\alpha_d,d}(\bm{x};\bm{z},t) \Phi_{\alpha_d,d}(\bm{y};\bm{z},t) \, \mathrm{d}\sigma_d(\bm{z}) \, \mathrm{d}t \\
&\phantom{equals}=
\frac{1}{\alpha_d} \left( 1 - C_d\|\bm{x}-\bm{y}\| - 1 + C_dI_d +  ( \alpha_d - C_d I_d ) \right)\\
&\phantom{equals}=
 1   - \frac{C_d}{\alpha_d}\|\bm{x}-\bm{y}\| = K_{\alpha_d,d}(\bm{x},\bm{y}).
\end{split}
\end{equation*}
This finishes the proof of \eqref{eq:distance-kernel-feature-integral}. Formula \eqref{eq:distance-kernel-centred-caps} follows from simple rearrangement.
\end{proof}

\begin{remark}\rm\label{rm_rk}
Since
\[
K_{\alpha_d,d}
=
\frac{1}{\alpha_d} K_{1,d} + 1 - \frac{1}{\alpha_d},
\]
all nonconstant spherical-harmonic coefficients are positive multiples of
those of $K_{1,d}$.  The constant coefficient equals
$1-C_dI_d/\alpha_d=(\alpha_d-C_dI_d)/\alpha_d>0$.  Hence
$K_{\alpha_d,d}$ is a reproducing kernel for
$H^{(d+1)/2}(\mathbb{S}^d)$, with a norm equivalent to the one induced by
$K_{1,d}$.  In particular, $\alpha_d=\sqrt2\,C_d$ is admissible and gives
\eqref{Ksqrt2Cd}; the required inequality $\sqrt2\,C_d>C_dI_d$, i.e., $I_d<\sqrt2$, follows from item~1 of Lemma~\ref{lem:C.d.I.d.inequalities}.
\end{remark}

For an $N$-point set $P_{N,d}$, define its distance-kernel discrepancy by
the squared worst-case error identity below.  Rotational invariance makes the
kernel mean constant, so the standard three-term worst-case error formula
reduces to the first line.  We retain the superscript $\cp$ to emphasize its
agreement with the classical cap discrepancy at $\alpha_d=1$.
Put
\begin{equation}\label{eq:Stolarsky.gen.A}
\begin{split}
(L_2^{\cp,\alpha_d}(P_{N,d}))^2
&:=\frac1{N^2}\sum_{m,n=1}^N
K_{\alpha_d,d}(\bm{x}_m,\bm{x}_n)
-\int_{\mathbb{S}^d}\int_{\mathbb{S}^d}
K_{\alpha_d,d}(\bm{x},\bm{y})
\,\mathrm{d}\sigma_d(\bm{x})\,\mathrm{d}\sigma_d(\bm{y})\\
&=\frac{C_d}{\alpha_d}\left(
I_d-\frac1{N^2}\sum_{m,n=1}^N
\|\bm{x}_m-\bm{x}_n\|\right).
\end{split}
\end{equation}
For $\alpha_d=1$, this is exactly the classical spherical cap $L_2$ discrepancy by
Stolarsky's invariance principle.  For every admissible $\alpha_d$, the
minimizers are therefore precisely the point sets that maximize the sum of
pairwise distances.

\paragraph{Stolarsky's invariance principle.} Let $\mathcal H_{\alpha_d,d}$ denote the reproducing kernel Hilbert space with kernel $K_{\alpha_d,d}$. For $N$-element sets $P_{N,d}=\{\bm{x}_1,\ldots,\bm{x}_N\}$ in $\mathbb{S}^d$ define the worst-case error for quasi-Monte Carlo integration in $\mathcal H_{\alpha_d,d}$ by
\[
e_{\alpha_d,d}(P_{N,d})
:=\sup_{\|f\|_{\mathcal H_{\alpha_d,d}}\leq1}
\left|\frac1N\sum_{n=1}^N f(\bm{x}_n)-\int_{\mathbb{S}^d}f(\bm{x})\,\mathrm{d}\sigma_d(\bm{x})\right|.
\]
The initial error $e_{\alpha_d,d}(0)$ is the error of the zero rule, not that of a nonempty equal-weight rule. Equivalently, the initial error can be considered as the norm of the integration functional.

\begin{theorem}[Stolarsky's identity for the distance-kernel]
\label{thm:distance-stolarsky}
For $d,N\in\mathbb{N}$ and $\alpha_d>C_dI_d$, 
\begin{align}
e_{\alpha_d,d}(P_{N,d})^2
&=\frac1{\alpha_d}\bigl(L_2^{\cp}(P_{N,d})\bigr)^2=\frac{C_d}{\alpha_d}
\left(I_d-\frac1{N^2}\sum_{m,n=1}^N
\|\bm{x}_m-\bm{x}_n\|\right).
\label{eq:Stolarsky.gen}
\end{align}
Equivalently,
\begin{equation*}
e_{\alpha_d,d}(P_{N,d})^2
+\frac{C_d}{\alpha_d} \, \frac{1}{N^2} \sum_{m,n=1}^N\|\bm{x}_m-\bm{x}_n\|
=\frac{C_dI_d}{\alpha_d}.
\end{equation*}
The initial squared error is
\begin{equation*}
e_{\alpha_d,d}(0)^2=g_{\alpha_d,d}
=1-\frac{C_dI_d}{\alpha_d}.
\end{equation*}
\end{theorem}

\begin{proof}
The squared RKHS worst-case error is the kernel energy of the signed
measure $\frac{1}{N} \sum_{n=1}^N \delta_{\bm{x}_n}-\sigma_d$. Substitution of the
representation \eqref{eq:distance-kernel-feature-integral}
and Fubini's theorem give
\begin{align*}
e_{\alpha_d,d}(P_{N,d})^2
&=\int_{-1}^{1}\int_{\mathbb{S}^d}
\left(\frac1N\sum_{n=1}^N
\Phi_{\alpha_d,d}(\bm{x}_n;\bm{z},t)
-\int_{\mathbb{S}^d}\Phi_{\alpha_d,d}(\bm{x};\bm{z},t)
\,\mathrm{d}\sigma_d(\bm{x})\right)^2
\,\mathrm{d}\sigma_d(\bm{z})\,\mathrm{d}t.
\end{align*}
For fixed $(\bm{z},t)$, the spherical mean of
$\Delta(\bm{x}; \bm{z}; t )$ with respect to $\bm{x}$ is zero. Consequently
\[
\int_{\mathbb{S}^d}\Phi_{\alpha_d,d}(\bm{x};\bm{z},t)
\,\mathrm{d}\sigma_d(\bm{x})=\sqrt{g_{\alpha_d,d}/2}.
\]
The error in the preceding integral is therefore
\[
\frac1{\sqrt{\alpha_d}}
\left(\frac1N\sum_{n=1}^N
\bm{1}_{C(\bm{z};t)}(\bm{x}_n)-\sigma_d(C(\bm{z};t))\right).
\]
This proves the first equality in \eqref{eq:Stolarsky.gen}; the
second follows from \eqref{eq:Stolarsky.principle}.
Alternatively, rotational invariance gives the constant kernel mean
\begin{equation}\label{eq:distance-kernel-mean}
\int_{\mathbb{S}^d}K_{\alpha_d,d}(\bm{x},\bm{y})
\,\mathrm{d}\sigma_d(\bm{y})=g_{\alpha_d,d}.
\end{equation}
Inserting \eqref{Kalpha} and \eqref{eq:distance-kernel-mean} into the
usual RKHS error formula proves the same identity directly. Finally,
the squared initial error is the double integral of the kernel,
which equals $g_{\alpha_d,d}$ by \eqref{eq:distance-kernel-mean}.
\end{proof}

\begin{remark}\rm
The identity \eqref{eq:Stolarsky.gen} holds for nonempty
equal-weight rules because their total weight is one. It does not
identify their initial errors. Indeed,
\[
\frac1{\alpha_d}\bigl(L_2^{\cp}(\emptyset)\bigr)^2
=\frac{1-C_dI_d}{\alpha_d},
\qquad
e_{\alpha_d,d}(0)^2=\frac{\alpha_d-C_dI_d}{\alpha_d},
\]
which agree only when $\alpha_d=1$. Every nonempty equal-weight rule
integrates constants exactly, whereas the zero rule does not.
In particular, for fixed $d$ and $N$, all the distance kernels have
the same optimal node sets: those maximizing
$\sum_{m,n}\|\bm{x}_m-\bm{x}_n\|$. Their normalized information
complexities nevertheless differ through their initial errors.
\end{remark}

For the balanced choice $\alpha_d=\sqrt2C_d$, the identities take
the particularly simple form
\begin{equation}\label{eq:balanced-stolarsky}
\sqrt2\,e_{\sqrt2C_d,d}(P_{N,d})^2
+\frac1{N^2}\sum_{m,n=1}^N\|\bm{x}_m-\bm{x}_n\|=I_d,
\end{equation}
with initial squared error $1-I_d/\sqrt2$.

\paragraph{Normalized information-complexity.}
The initial discrepancy is
\begin{equation}\label{init:err:dist:ker}
(L_2^{\cp,\alpha_d}(\emptyset))^2
:=\int_{\mathbb{S}^d}\int_{\mathbb{S}^d}
K_{\alpha_d,d}(\bm{x},\bm{y})
\,\mathrm{d}\sigma_d(\bm{x})\,\mathrm{d}\sigma_d(\bm{y})
=1-\frac{C_dI_d}{\alpha_d}.
\end{equation}

\begin{definition}\rm
For $\varepsilon\in(0,1)$ and $d \in \mathbb{N}$, the normalized information-complexity of the distance-kernel $L_2$ discrepancy is
\begin{equation*}
N^{\cp,\alpha_d}(\varepsilon,d)
:=\inf\big\{N\in\mathbb{N}:\exists P_{N,d}\subseteq\mathbb{S}^d
\text{ such that }L_2^{\cp,\alpha_d}(P_{N,d})
\leq\varepsilon L_2^{\cp,\alpha_d}(\emptyset)\big\}.
\end{equation*}
\end{definition}

\begin{theorem}\label{thm1}
For every $\varepsilon\in(0,1)$, $d\in\mathbb{N}$, and
$\alpha_d > C_dI_d$, we have
\begin{equation}\label{Ned_alpha}
\begin{aligned}
\max\left\{\begin{gathered}
\frac1{10}\left(
\frac1{\sqrt d}\,\frac1{\varepsilon^2}\,
\frac{C_dI_d}{\alpha_d-C_dI_d}
\right)^{\!\tfrac{d}{d+1}},\\[-1mm]
\left(
\frac1{d+1}
+4\sqrt2\,\varepsilon^2\frac{\alpha_d-C_dI_d}{C_d}
\right)^{-1}
\end{gathered}\right\}
&\leq N^{\cp,\alpha_d}(\varepsilon,d),\\
N^{\cp,\alpha_d}(\varepsilon,d)
&\leq
\left\lceil
\frac1{\varepsilon^2}\frac{C_dI_d}{\alpha_d-C_dI_d}
\right\rceil.
\end{aligned}
\end{equation}
\end{theorem}

\begin{remark} \rm
The lower bound in \eqref{Ned_alpha} is due to explicit (with respect to $N$ and $d$) lower bounds for the spherical cap $L_2$ discrepancy obtained in \cite{BB25}; see also Appendix~\ref{app:lower.bound}.
By Stolarsky's invariance principle \eqref{eq:Stolarsky.gen.A}, a matching upper bound for the discrepancy requires a good \emph{lower} bound on the maximal achievable sum of pairwise distances, that is, a bound of the form
\[
\max_{P_{N,d}\subseteq\mathbb{S}^d}\ \sum_{m,n=1}^N\|\bm{x}_m-\bm{x}_n\|
\geq N^2 I_d - c_d\, N^{1-\frac{1}{d}}
\]
with an explicit constant $c_d$, preferably obtained by an explicit construction of the point sets. Upper bounds for sums of distances, such as Lemma~\ref{lem:mean-chordal-simplex}, lead to lower bounds for the discrepancy instead. Since we are not aware of published bounds of this type with constants explicit in $N$ and $d$, we use a probabilistic argument for the upper bound in \eqref{Ned_alpha}. It yields the above estimate with $N I_d$ in place of $c_d N^{1-1/d}$, and it does not provide explicit point sets. Maximal sums of pairwise distances, and hence, by Stolarsky's invariance principle, the spherical cap $L_2$ discrepancy, are closely related to minimal (Riesz) energy problems; we refer to \cite{BHS2019,BHS2012}. A conjecture for the optimal sum of distances for large $N$ (and fixed $d$) can be found in \cite{B2011}.
\end{remark}

\begin{proof}[Proof of Theorem~\ref{thm1}]
For independent, uniformly distributed points, the average squared
distance-kernel discrepancy satisfies
\begin{align*}
\mathbb{E}\left[(L_2^{\cp,\alpha_d}(P_{N,d}))^2\right] & = \frac{1}{N} \left( \int_{\mathbb{S}^d} K_{\alpha_d, d}(\bm{x}, \bm{x}) \,\mathrm{d} \sigma_d(\bm{x}) - \int_{\mathbb{S}^d} \int_{\mathbb{S}^d} K_{\alpha_d, d}(\bm{x}, \bm{y}) \,\mathrm{d} \sigma_d(\bm{x}) \, \mathrm{d} \sigma_d(\bm{y}) \right) \\ & = \frac{1}{N} \left( 1 - 1 + \frac{C_d I_d}{\alpha_d}  \right) = \frac{C_d I_d}{\alpha_d N}.
\end{align*}
As in the proof of Theorem~\ref{thm:1}, this shows that
\begin{equation*}
N^{\cp,\alpha_d}(\varepsilon, d) \le \left\lceil \frac{1}{\varepsilon^2} \, \frac{C_d I_d}{\alpha_d - C_d I_d }\right\rceil.
\end{equation*}
This proves the upper bound in \eqref{Ned_alpha}.

The first lower term in \eqref{Ned_alpha} follows from the general lower bound proved in Appendix~\ref{app:lower.bound}: for any $N$-point set $P_{N,d}$ on $\mathbb{S}^d$ we have  
\begin{equation}\label{L2_lower_bound}
(L_2^{\cp,\alpha_d}(P_{N,d}))^2  \geq 
\frac{C_d I_d}{\alpha_d}  \, \rho_0\frac{F(x_d^*)}{2 \sqrt{2\pi}} \, \frac{1}{N^{(d+1)/d}},
\end{equation}
where (cf. Eq.~\eqref{eq:F.x.star})
\begin{equation*}
F(x_d^*) = \frac{d}{(d+1)^{\tfrac{d+1}{d}}} \left(\frac{\Gamma(\frac{1}{2})}{\Gamma(\frac{d+1}{2})}\right)^{\tfrac{1}{d}}
\qquad 
\mbox{and} \qquad \rho_0:=\sqrt{\frac4{4+\pi}}.
\end{equation*}
This implies that 
\begin{equation*}
N^{\cp,\alpha_d}(\varepsilon, d) 
\geq 
\left( \frac{1}{\varepsilon^2} \, \frac{C_d \, I_d}{\alpha_d - C_d \, I_d} \right)^{\tfrac{d}{d+1}} \left(\rho_0\frac{F(x_d^*)}{2 \sqrt{2 \pi}} \right)^{\tfrac{d}{d+1}}.
\end{equation*}
We use the bound
$\Gamma(x)<\sqrt{2\pi}\,x^{x-1/2}{\rm e}^{-x}{\rm e}^{1/(12x)}$ for
$x>0$ from \cite[Eq.~(5.6.1)]{NIST}, together with
$\Gamma(\tfrac{1}{2})=\sqrt\pi$, to obtain
\begin{align*}
\left( \frac{F(x_d^*)}{2 \sqrt{2\pi}} \right)^{\tfrac{d}{d+1}}
&=
\frac{d^{\tfrac{d}{d+1}}}{d+1} \left( \frac{1}{2 \sqrt{2 \pi}} \right)^{\tfrac{d}{d+1}} \left(\frac{\Gamma(\frac{1}{2})}{\Gamma(\frac{d+1}{2})}\right)^{\tfrac{1}{d+1}}
\\
&>
\frac{1}{2 \sqrt{2 \pi}} \, \frac{d}{d+1} \left( \frac{2 \, \sqrt{2\pi}}{d} \right)^{\tfrac{1}{d+1}} \left(\frac{\sqrt{\pi}}{\sqrt{2\pi} \left( \frac{d+1}{2} \right)^{\tfrac{d}{2}} {\rm e}^{-\tfrac{d+1}{2}} \, {\rm e}^{\tfrac{1}{6(d+1)}}}\right)^{\tfrac{1}{d+1}}
\\
&= \frac{\sqrt{{\rm e}}}{2 \, \sqrt{\pi}} \left( \frac{1}{\sqrt{d}} \right)^{\tfrac{d}{d+1}}  f(d) \, h(d),
\end{align*}
where 
\begin{equation*}
f(d) := \left( \frac{\sqrt{2 \pi}}{d} \sqrt{1 + \frac{1}{d}}\right)^{\tfrac{1}{d+1}} \quad \mbox{ and } \quad h( d ) := \frac{{\rm e}^{-\tfrac{1}{6(d+1)^2}}}{\left( 1 + \frac{1}{d} \right)^{\tfrac{3}{2}}}.
\end{equation*}
Differentiating $\log f$ shows that $f$ decreases up to its unique
minimum in $(8,9)$ and increases thereafter.  Since $f(9)<f(8)$, it follows
that $f(d)\geq f(9)$ for every $d\in\mathbb{N}$.  Likewise, direct
differentiation shows that $h$ is strictly increasing on $(0,\infty)$, so
$h(d)\geq h(1)$ for every $d\in\mathbb{N}$.  Thus
\begin{equation*}
N^{\cp,\alpha_d}(\varepsilon, d) 
\geq 
c\,\rho_0^{\tfrac{d}{d+1}}
\left( \frac{1}{\sqrt{d}} \, \frac{1}{\varepsilon^2} \, \frac{C_d \, I_d}{\alpha_d - C_d \, I_d} \right)^{\tfrac{d}{d+1}},
\end{equation*}
where 
\begin{equation*}
c := \frac{\sqrt{{\rm e}}}{2 \, \sqrt{\pi}} \, f(9) \, h(1) 
= 
\frac{5^{1/20} {\rm e}^{11/24}}{4 \cdot 2^{2/5} 3^{3/10} \pi^{9/20}}
=
0.1395\ldots.
\end{equation*}
Since $0<\rho_0<1$ and $\tfrac{d}{d+1}<1$, the estimate needed for the first
lower term in \eqref{Ned_alpha} is
\[
c\,\rho_0^{\tfrac{d}{d+1}}>c\rho_0>0.104>\frac1{10}.
\]

It remains to prove the second, dimension-sharp lower term. Our starting point is the representation \eqref{eq:L2.cap.a.d.A} in Appendix~\ref{app:lower.bound}. 
Since all terms are non-negative, retaining only $\ell = 2$ gives 
(cf. \cite[Proposition~4]{BB25})
\begin{align}
\label{eq:degree-two-lower}
(L_2^{\cp,\alpha_d}(P_{N,d}))^2
&\geq
\frac{\sqrt2\,C_d}{8 \alpha_d} \, 
 \left( \frac{1}{N^2} \sum_{m,n=1}^N \langle \bm{x}_m, \bm{x}_n \rangle^2 - \int_{\mathbb{S}^d} \int_{\mathbb{S}^d} \langle \bm{x}, \bm{y} \rangle^2 \,\mathrm{d} \sigma_d(\bm{x}) \,\mathrm{d} \sigma_d(\bm{y}) \right) \notag \\
&=
\frac{\sqrt2\,C_d}{8\alpha_d} \left( M_2 - \frac{1}{d+1} \right), 
\qquad 
M_2 := \frac{1}{N^2} \sum_{m,n=1}^N \langle \bm{x}_m, \bm{x}_n \rangle^2.
\end{align}
To estimate $M_2$, consider the symmetric $(d+1)\times (d+1)$-matrix $S:=\sum_{n=1}^N\bm{x}_n\bm{x}_n^{\top}$.  Then $\operatorname{tr}S=N$ and
\[
M_2=\frac{\operatorname{tr}(S^2)}{N^2}
\geq\frac{(\operatorname{tr}S)^2}
{N^2\operatorname{rank}S}
\geq\frac1{\min\{N,d+1\}},
\]
where the first inequality is Cauchy-Schwarz for the nonzero eigenvalues
of $S$ (of which there are exactly $\operatorname{rank}S$ many since $S$ is symmetric). Suppose that $P_{N,d}$ satisfies the normalized error criterion $L_2^{\cp,\alpha_d}(P_{N,d}) \leq\varepsilon L_2^{\cp,\alpha_d}(\emptyset)$. If $N\leq d+1$, the initial-error formula above and \eqref{eq:degree-two-lower} yield
\begin{equation}
\label{eq:frame.potential.inequ}
\frac{\sqrt2\,C_d}{8\alpha_d}
\left(\frac1N-\frac1{d+1}\right)
\leq
\varepsilon^2 \, \frac{\alpha_d-C_dI_d}{\alpha_d}.
\end{equation}
Solving for $N$ gives
\[
N\geq
\left(
\frac1{d+1}
+4\sqrt2\,\varepsilon^2\frac{\alpha_d-C_dI_d}{C_d}
\right)^{-1}.
\]
If $N > d + 1$, the bound holds trivially, because its right-hand side is smaller than $d+1$. This proves the second lower term and completes the proof.
\end{proof}

\begin{corollary}\label{cor:balanced-linear}
For the balanced choice $\alpha_d=\sqrt2\,C_d$ and all
$\varepsilon\in(0,1)$,
\begin{equation} \label{eq:balanced-exact-bounds}
\left(
\frac1{d+1}+8 \, \varepsilon^2
\left(1-\frac{I_d}{\sqrt2}\right)
\right)^{-1}
\leq N^{\cp,\sqrt2C_d}(\varepsilon,d)
\leq
\left\lceil
\frac{1}{\varepsilon^2} \, \frac{I_d}{\sqrt2-I_d}
\right\rceil.
\end{equation}
Consequently, for every fixed $\varepsilon\in(0,1)$,
\begin{equation}\label{Nsqrt2Cd}
\frac{d}{1+\varepsilon^2}\,(1+o(1))
\leq N^{\cp,\sqrt2C_d}(\varepsilon,d)
\leq\frac{8d}{\varepsilon^2}\,(1+o(1))
\qquad \mbox{as } d\to\infty,
\end{equation}
and hence $N^{\cp,\sqrt2C_d}(\varepsilon,d)=\Theta_\varepsilon(d)$.
\end{corollary}

\begin{proof}
We apply Theorem~\ref{thm1} with $\alpha_d=\sqrt2\,C_d$ to get \eqref{eq:balanced-exact-bounds}.  
The asymptotic relation \eqref{Nsqrt2Cd} follows from the asymptotic expansion in item 1 of Lemma~\ref{lem:C.d.I.d.inequalities}, since  
\begin{equation*}
1 - \frac{I_d}{\sqrt2} = \frac{1}{8d} + \mathcal{O}\left( \frac{1}{d^2} \right) 
\qquad 
\text{and} 
\qquad 
\frac{I_d}{\sqrt2-I_d} = 8d \big( 1 + o( 1 ) \big) 
\end{equation*}
and 
\begin{equation*}
\frac1{d+1} + 8 \, \varepsilon^2 \left(1-\frac{I_d}{\sqrt2}\right) 
=
\frac{1}{d} \left( 1 + \mathcal{O}\Big( \frac{1}{d} \Big) \right) + \frac{\varepsilon^2}{d} \left( 1 + \mathcal{O}\Big( \frac{1}{d} \Big) \right) 
= \frac{1+\varepsilon^2}{d} \left( 1 + \mathcal{O}\Big( \frac{1}{d} \Big) \right) 
\end{equation*}
as $d \to \infty.$
\end{proof}

\paragraph{Discussion.}
Theorem~\ref{thm1} immediately gives the blessing of dimensionality whenever
$\alpha_d-C_dI_d$ is bounded below by a positive absolute constant, because
$C_dI_d=\mathcal O(d^{-1/2})$.  To describe the more delicate regimes, let
$g:\mathbb{N}\to(0,1)$ and set
\begin{equation*}
\alpha_d:=\frac{C_dI_d}{1-g(d)}.
\end{equation*}
Then $\alpha_d>C_dI_d$ and $(L_2^{\cp,\alpha_d}(\emptyset))^2=g(d)$.

Substitution in Theorem~\ref{thm1} yields
\begin{equation*}
\begin{split}
\max\Bigg\{&
\frac1{10}\left(
\frac1{\sqrt d\,\varepsilon^2}\frac{1-g(d)}{g(d)}
\right)^{\!\tfrac{d}{d+1}},\\[-1mm]
&\left(
\frac1{d+1}
+4\sqrt2\,\varepsilon^2I_d\frac{g(d)}{1-g(d)}
\right)^{-1}\Bigg\}
\leq N^{\cp,\alpha_d}(\varepsilon,d)
\leq
\left\lceil
\frac1{\varepsilon^2}\frac{1-g(d)}{g(d)}
\right\rceil.
\end{split}
\end{equation*}
These bounds give the following consistent picture.
\begin{itemize}
\item If $g(d)\to1$, then $(1-g(d))/g(d)\to0$ and the random-point upper
bound gives the blessing of dimensionality.  The classical cap discrepancy,
for which $\alpha_d=1$ and $g(d)=1-C_dI_d$, belongs to this regime.

\item If $g(d)\asymp d^{-\delta}$ with $\delta\in (0,1]$, then for every
fixed $\varepsilon\in(0,1)$ the second lower bound and the upper bound match:
\[
N^{\cp,\alpha_d}(\varepsilon,d)=\Theta_\varepsilon(d^\delta).
\]
Thus no such family is strongly polynomially tractable.  The balanced choice
$g(d)=1-I_d/\sqrt2\asymp d^{-1}$, or equivalently,
$\alpha_d=\sqrt2\,C_d$, is the endpoint $\delta=1$ and is covered sharply by
Corollary~\ref{cor:balanced-linear}. As already mentioned in \eqref{Ksqrt2Cd}, its kernel is $K_{\sqrt2C_d,d}(\bm{x},\bm{y})=1-\frac1{\sqrt2}\|\bm{x}-\bm{y}\|$ for $\bm{x},\bm{y}\in\mathbb{S}^d$.

\item If $g(d)\asymp d^{-\delta}$ with $\delta>1$, the result still gives
polynomial tractability but does not in general determine the exact power of
$d$: the frame-potential term is of order $d$, the earlier lower term is of
order $d^{\delta-1/2+o(1)}$, and the upper bound is of order $d^\delta$ for
fixed $\varepsilon$.  

\item If $g(d)\asymp a^d$ for some $a\in(0,1)$, the first
lower term and the upper bound are exponential, so the curse of dimensionality
occurs.

\item If $g(d)\to c\in(0,1)$, then the upper bound is uniform in $d$, and
the problem is strongly polynomially tractable.  A concrete choice is
$\alpha_d=\eta C_d$ with fixed $\eta>\sqrt2$, for which
$K_{\eta C_d,d}(\bm{x},\bm{y})=1-\eta^{-1}\|\bm{x}-\bm{y}\|$ and
$g(d)\to1-\sqrt2/\eta$.  The normalized one-sided kernel in
Section~\ref{sec:one-sided} corresponds to $\eta=2$.
\end{itemize}

\section{A one-sided spherical cap \texorpdfstring{$L_2$}{L2} discrepancy}\label{sec:one-sided}

We now retain only nonnegative cap heights.  Thus caps larger than a hemisphere
are excluded; for every fixed negative height, their measures tend to one as
$d\to\infty$ by concentration of measure.
\begin{definition}\rm
For $d \in \mathbb{N}$, the one-sided spherical cap $L_2$ discrepancy of an $N$-point set
$P_{N,d}=\{\bm{x}_1,\ldots,\bm{x}_N\}\subseteq\mathbb{S}^d$ is defined by
\begin{equation}\label{def:L2-one-sided}
L_2^{\cp,+}(P_{N,d})
:=\left(\int_0^1\int_{\mathbb{S}^d}
\left(
\frac1N\sum_{n=1}^N\bm{1}_{C(\bm{z};t)}(\bm{x}_n)
-\sigma_d(C(\bm{z};t))
\right)^2\,\mathrm{d}\sigma_d(\bm{z})\,\mathrm{d}t\right)^{1/2}.
\end{equation}
\end{definition}

The associated kernel is
\begin{equation}\label{def:K-one-sided}
K_{+,d}(\bm{x},\bm{y})
:=\int_0^1\int_{\mathbb{S}^d}
\bm{1}_{C(\bm{z};t)}(\bm{x})
\bm{1}_{C(\bm{z};t)}(\bm{y})
\,\mathrm{d}\sigma_d(\bm{z})\,\mathrm{d}t,
\qquad \bm{x},\bm{y}\in\mathbb{S}^d.
\end{equation}
This representation shows directly that $K_{+,d}$ is positive
definite.

\begin{theorem}\label{thm:one-sided-kernel}
For every $d\in\mathbb{N}$,
\begin{equation}\label{eq:one-sided-kernel}
K_{+,d}(\bm{x},\bm{y})
=C_d\left(1-\frac12\|\bm{x}-\bm{y}\|\right), \qquad \bm{x},\bm{y}\in\mathbb{S}^d.
\end{equation}
Consequently, up to the factor $C_d$, this kernel is the distance kernel
$K_{2C_d,d}$ of \eqref{Kalpha}. In particular, its RKHS is the Sobolev space
$H^{(d+1)/2}(\mathbb{S}^d)$ as a set, equipped with an equivalent norm.
\end{theorem}

\begin{proof}
Fix $\bm{x},\bm{y}\in\mathbb{S}^d$ and put $a:=\langle\bm{z},\bm{x}\rangle$ and $b:=\langle\bm{z},\bm{y}\rangle$. For fixed $\bm{z}$, the inner integral with respect to $t$ in \eqref{def:K-one-sided} is
\begin{equation}\label{eq:min-positive-part}
\int_0^1\bm{1}_{\{a\geq t\}}\bm{1}_{\{b\geq t\}}\,\mathrm{d}t
=(\min\{a,b\})_+.
\end{equation}
Here, $(x)_+ = x$ if $x \geq 0$ and $(x)_+ = 0$ if $x < 0$. By the Funk-Hecke formula \cite{M06}, writing $\omega_d$ for the
surface area of $\mathbb{S}^d$, we get
\begin{equation}
\label{eq:positive-coordinate-one-sided}
\int_{\mathbb{S}^d}
(\langle\bm{z},\bm{x}\rangle)_+
\,\mathrm{d}\sigma_d(\bm{z})
= 
\frac{\omega_{d-1}}{\omega_d} \int_0^1 \tau (1-\tau^2)^{\frac{d}{2}-1} \mathrm{d} \tau
=
\frac{\Gamma(\tfrac{d+1}{2})}{\sqrt{\pi} \, \Gamma( \frac{d}{2} )} \, \frac1d
=
C_d.
\end{equation}

Let $K_{-,d}$ denote the corresponding kernel for the negative half of
the height interval. Substituting $t=-s$, replacing $\bm{z}$ by $-\bm{z}$,
and using inclusion-exclusion gives, up to sets of measure zero,
\begin{align*}
K_{-,d}(\bm{x},\bm{y})
&=\int_0^1\int_{\mathbb{S}^d}
\bm{1}_{\{a\leq s\}}\bm{1}_{\{b\leq s\}}
\,\mathrm{d}\sigma_d(\bm{z})\,\mathrm{d}s\\
&=1-\int_0^1\int_{\mathbb{S}^d}\bm{1}_{\{a>s\}}
\,\mathrm{d}\sigma_d(\bm{z})\,\mathrm{d}s
-\int_0^1\int_{\mathbb{S}^d}\bm{1}_{\{b>s\}}
\,\mathrm{d}\sigma_d(\bm{z})\,\mathrm{d}s
+K_{+,d}(\bm{x},\bm{y})\\
&=1-2C_d+K_{+,d}(\bm{x},\bm{y}),
\end{align*}
where the last equality follows from
\eqref{eq:positive-coordinate-one-sided}. Since
$K_{1,d}=K_{-,d}+K_{+,d}$ and since $K_{1,d}(\bm{x},\bm{y}) =1-C_d\|\bm{x}-\bm{y}\|$ by \eqref{def:ker}, solving for $K_{+,d}$ gives the formula in \eqref{eq:one-sided-kernel}. 

To identify the RKHS, note that
\[
\frac{1}{C_d} K_{+,d}(\bm{x},\bm{y})
=1-\frac12\|\bm{x}-\bm{y}\|=K_{2C_d,d}(\bm{x},\bm{y}),
\]
where $K_{2C_d,d}$ is the kernel from \eqref{Kalpha} with $\alpha_d=2C_d$. 
The choice $\alpha_d=2C_d$ satisfies $\alpha_d > C_dI_d$, since $I_d<\sqrt2<2$ by item~1 of Lemma~\ref{lem:C.d.I.d.inequalities}. The assertion about the RKHS now follows
from Remark~\ref{rm_rk}. Multiplication of a kernel by a positive constant
does not change its set of functions, but only rescales its norm.
\end{proof}

The next result is the one-sided form of Stolarsky's invariance principle.

\begin{theorem}
\label{thm:one-sided-stolarsky}
For every $P_{N,d}=\{\bm{x}_1,\ldots,\bm{x}_N\}\subseteq\mathbb{S}^d$ we have
\begin{equation}\label{eq:one-sided-stolarsky}
(L_2^{\cp,+}(P_{N,d}))^2
=\frac{C_d}{2}\left(
I_d-\frac1{N^2}\sum_{m,n=1}^N\|\bm{x}_m-\bm{x}_n\|
\right).
\end{equation}
Equivalently,
\begin{equation}\label{eq:one-sided-stolarsky-invariant}
\frac1{N^2}\sum_{m,n=1}^N\|\bm{x}_m-\bm{x}_n\|
+\frac{2}{C_d}\big(L_2^{\cp,+}(P_{N,d})\big)^2=I_d.
\end{equation}
Moreover,
\begin{equation}\label{eq:one-sided-half-classical}
(L_2^{\cp,+}(P_{N,d}))^2 = \frac12 (L_2^{\cp}(P_{N,d}))^2.
\end{equation}
\end{theorem}

\begin{proof}
Expanding the square in \eqref{def:L2-one-sided} and applying Fubini's theorem
gives
\begin{align*}
(L_2^{\cp,+}(P_{N,d}))^2
&=\frac1{N^2}\sum_{m,n=1}^N K_{+,d}(\bm{x}_m,\bm{x}_n)
-\frac2N\sum_{n=1}^N\int_{\mathbb{S}^d}
K_{+,d}(\bm{x}_n,\bm{y})\,\mathrm{d}\sigma_d(\bm{y})\\
&\quad+\int_{\mathbb{S}^d}\int_{\mathbb{S}^d}
K_{+,d}(\bm{x},\bm{y})
\,\mathrm{d}\sigma_d(\bm{x})\,\mathrm{d}\sigma_d(\bm{y}).
\end{align*}
Theorem~\ref{thm:one-sided-kernel} and rotational invariance show that the
kernel mean is independent of $\bm{x}$ and equals
\begin{equation}\label{eq:one-sided-kernel-mean}
\int_{\mathbb{S}^d}K_{+,d}(\bm{x},\bm{y})
\,\mathrm{d}\sigma_d(\bm{y})
=C_d\left(1-\frac{I_d}{2}\right).
\end{equation}
The double integral of the kernel has the same value. Substitution of
\eqref{eq:one-sided-kernel} and \eqref{eq:one-sided-kernel-mean} proves
\eqref{eq:one-sided-stolarsky}, and rearrangement gives
\eqref{eq:one-sided-stolarsky-invariant}. Finally, let
\[
\Delta_{P_{N,d}}(\bm{z},t):=\frac{1}{N} \sum_{n=1}^N \bm{1}_{C(\bm{z}; t)}(\bm{x}_n) - \sigma_d(C(\bm{z}; t)), \qquad \bm{z}\in\mathbb{S}^d,\ t\in[-1,1],
\]
denote the local cap discrepancy of $P_{N,d}$. The sets $C(-\bm{z};t)$ and $\mathbb{S}^d \setminus C(\bm{z};-t)=\{\bm{y}\in\mathbb{S}^d:\langle -\bm{z},\bm{y}\rangle>t\}$ differ only by the boundary $\{\bm{y}\in\mathbb{S}^d:\langle \bm{z},\bm{y}\rangle=-t\}$. This set has $\sigma_d$-measure zero, and for almost every $(\bm{z},t)$ it contains none of the points $\bm{x}_1,\ldots,\bm{x}_N$. Since both the empirical measure $\frac1N\sum_{n=1}^N\delta_{\bm{x}_n}$ and $\sigma_d$ have total mass one, it follows that $\Delta_{P_{N,d}}(\bm{z},-t)=-\Delta_{P_{N,d}}(-\bm{z},t)$ for almost every $(\bm{z},t)\in\mathbb{S}^d\times[0,1]$. Integrating the squares and using, in the last step, the invariance of $\sigma_d$ under $\bm{z}\mapsto-\bm{z}$ yields
\begin{align*}
(L_2^{\cp}(P_{N,d}))^2 = & \int_{-1}^1 \int_{\mathbb{S}^d} |\Delta_{P_{N,d}}(\bm{z},t)|^2 \,\mathrm{d}\sigma_d(\bm{z}) \mathrm{d} t \\
= & \int_{0}^1 \int_{\mathbb{S}^d} |\Delta_{P_{N,d}}(\bm{z},t)|^2 \,\mathrm{d}\sigma_d(\bm{z}) \mathrm{d} t + \int_{-1}^0 \int_{\mathbb{S}^d} |\Delta_{P_{N,d}}(\bm{z},t)|^2 \,\mathrm{d}\sigma_d(\bm{z}) \mathrm{d} t\\
= & (L_2^{\cp,+}(P_{N,d}))^2 + \int_0^1 \int_{\mathbb{S}^d} |\Delta_{P_{N,d}}(\bm{z},-t)|^2 \,\mathrm{d}\sigma_d(\bm{z}) \mathrm{d} t\\
= & (L_2^{\cp,+}(P_{N,d}))^2 + \int_0^1 \int_{\mathbb{S}^d} |\Delta_{P_{N,d}}(-\bm{z},t)|^2 \,\mathrm{d}\sigma_d(\bm{z}) \mathrm{d} t\\
= & 2 (L_2^{\cp,+}(P_{N,d}))^2,
\end{align*}
which proves \eqref{eq:one-sided-half-classical}.
\end{proof}

\begin{remark}\rm
The relation \eqref{eq:one-sided-half-classical} holds for nonempty
equal-weight rules because both the empirical measure and $\sigma_d$ have total
mass one. It does not extend to the two initial errors. Indeed, defining the
one-sided initial error in the usual way gives
\begin{equation}\label{eq:one-sided-initial-error}
(L_2^{\cp,+}(\emptyset))^2 :=\int_0^1\int_{\mathbb{S}^d}
\sigma_d(C(\bm{z};t))^2\,\mathrm{d}\sigma_d(\bm{z})\,\mathrm{d}t
=C_d\left(1-\frac{I_d}{2}\right),
\end{equation}
whereas the classical initial squared error is $1-C_dI_d$ by
\eqref{init:err:fo}. Thus $L_2^{\cp,+}(\emptyset)\asymp d^{-1/4}$ as $d \to \infty$, which follows from Lemma~\ref{lem:C.d.I.d.inequalities}.
\end{remark}

\begin{definition}\rm
For $\varepsilon\in(0,1)$ and $d\in\mathbb{N}$, the normalized
information-complexity of the one-sided spherical cap $L_2$ discrepancy is 
\begin{equation*}
N^{\cp,+}(\varepsilon,d)
:=\inf\left\{N\in\mathbb{N}:\exists P_{N,d}\subseteq\mathbb{S}^d
\text{ such that }L_2^{\cp,+}(P_{N,d})
\leq\varepsilon L_2^{\cp,+}(\emptyset)\right\}.
\end{equation*}
\end{definition}

The following result can be obtained as a special case of \cite[Theorem~2.4.1]{BHS2019} and its proof.

\begin{lemma}
\label{lem:mean-chordal-simplex}
For every $d,N \in \mathbb{N}$ and every $P_{N,d}=\{\bm{x}_1,\ldots,\bm{x}_N\}\subseteq\mathbb{S}^d$,
\begin{equation}\label{eq:mean-chordal-simplex}
\frac1{N^2} \sum_{m,n=1}^N \|\bm{x}_m-\bm{x}_n \|
\leq \sqrt{2\left(1-\frac1N\right)}.
\end{equation}
For $N = 1$ equality is trivial. For $N\geq 2$, equality holds if and only if $P_{N,d}$ is a regular $(N-1)$-simplex inscribed in~$\mathbb{S}^d$ centered at the origin.
\end{lemma}
\begin{proof}
For the degenerate case $N=1$ the assertion is immediate, so assume $N \geq 2$. 
We adapt the proof of \cite[Theorem~2.4.1]{BHS2019} which considers a much larger class of kernels $K(\bm{x},\bm{y}) = f( \| \bm{x} - \bm{y} \|^2 )$. 
The general upper bound follows from a concavity argument for the function $f( x ) = \sqrt{x}$ (which is strictly concave and strictly increasing),  
\begin{equation*}
\sum_{m,n=1}^N \left\| \bm{x}_m - \bm{x}_n \right\|
\leq 
N \left( N - 1 \right) \sqrt{ \frac{\Lambda( P_{N,d} )}{N \left( N - 1 \right)} }
\leq 
N \left( N - 1 \right) \sqrt{ \frac{2N}{N-1} }
= N^2
\sqrt{2 \, \left( 1 - \frac{1}{N} \right) },
\end{equation*}
where it is used that
\begin{equation*}
\Lambda( P_{N,d} )  
:=
\sum_{m,n=1}^N \left\| \bm{x}_m - \bm{x}_n \right\|^2
=
2N^2 - 2\left\| \sum_{n=1}^N \bm{x}_n \right\|^2
\leq
2N^2.
\end{equation*}
The first inequality due to concavity becomes equality if and only if all $\left\| \bm{x}_m - \bm{x}_n \right\|$, $m \neq n$, are the same, and equality in the very last relation holds if and only if the centroid $\frac{1}{N}\sum_{n=1}^N \bm{x}_n$ is the origin. Both requirements are satisfied if and only if $P_{N,d}$ is a regular $(N-1)$-simplex inscribed in~$\mathbb{S}^d$ centered at the origin which exists exactly when $2 \leq N \leq d+2$. Indeed, since $\langle \bm{x}_m, \bm{x}_n \rangle = - \frac{1}{N-1}$, $m \neq n$, and thus $\left\| \bm{x}_m - \bm{x}_n \right\| = \sqrt{\frac{2N}{N-1}}$, we get 
\begin{equation*}
\frac{1}{N^2} \, \sum_{m,n=1}^N \left\| \bm{x}_m - \bm{x}_n \right\|
=
\frac{N \left( N - 1 \right)}{N^2} \, \sqrt{\frac{2N}{N-1}} 
=
\sqrt{2\left(1-\frac1N\right)} 
\geq 
\frac{1}{N^2} \, \sum_{m,n=1}^N \left\| \bm{\xi}_m - \bm{\xi}_n \right\|
\end{equation*} 
for any $P_{N,d} = \{ \bm{\xi}_1, \dots, \bm{\xi}_N \}\subseteq\mathbb{S}^d$ for $2 \leq N \leq d + 2$.
\end{proof}

\begin{theorem}\label{thm:one-sided-tractability}
For $\varepsilon\in(0,1)$ and $d\in\mathbb{N}$, set
\begin{equation}\label{eq:R-one-sided}
R_d(\varepsilon):=\left(
1-\frac12\left[I_d-\varepsilon^2(2-I_d)\right]^2
\right)^{-1}.
\end{equation}
Then
\begin{equation}\label{eq:N-one-sided-bounds}
\begin{aligned}
\max\left\{
\frac1{10}\left(
\frac1{\sqrt d}\,\frac1{\varepsilon^2}\,
\frac{I_d}{2-I_d}
\right)^{\tfrac{d}{d+1}},
\left\lceil R_d(\varepsilon)\right\rceil
\right\}
&\leq N^{\cp,+}(\varepsilon,d),\\
N^{\cp,+}(\varepsilon,d)
&\leq\left\lceil
\frac1{\varepsilon^2}\frac{I_d}{2-I_d}
\right\rceil\leq\left\lceil \frac{1+\sqrt2}{\varepsilon^2}\right\rceil.
\end{aligned}
\end{equation}
If $\lceil R_d(\varepsilon)\rceil\leq d+2$, then the dimension-sharp lower
bound is exact:
\begin{equation}\label{eq:N-one-sided-exact}
N^{\cp,+}(\varepsilon,d)=\left\lceil R_d(\varepsilon)\right\rceil.
\end{equation}
Consequently, the one-sided problem is strongly polynomially tractable, and
its optimal strong-tractability exponent, i.e., the number $$p^*:=\inf\{p \ge 0 : \exists C>0 \mbox{ with } N^{\cp,+}(\varepsilon,d) \le C \varepsilon^{-p} \ \forall \varepsilon \in (0,1),d \in \mathbb{N}\}$$ is $p^*=2$.
\end{theorem}

\begin{proof}
Let the points of $P_{N,d}$ be independent and identically uniformly distributed on
$\mathbb{S}^d$. Since $K_{+,d}(\bm{x},\bm{x})=C_d$, the standard random-point
calculation and \eqref{eq:one-sided-kernel-mean} give
\begin{equation}\label{eq:one-sided-random-error}
\mathbb{E}\left[(L_2^{\cp,+}(P_{N,d}))^2\right]
=\frac1N\left(C_d-C_d\left(1-\frac{I_d}{2}\right)\right)
=\frac{C_dI_d}{2N}.
\end{equation}
Hence there exists a deterministic point set with squared discrepancy no
larger than the right-hand side. Dividing by
\eqref{eq:one-sided-initial-error}, we find
\[
\frac{(L_2^{\cp,+}(P_{N,d}))^2}{(L_2^{\cp,+}(\emptyset))^2}
\leq\frac1N\frac{I_d}{2-I_d}.
\]
This proves the first upper bound in \eqref{eq:N-one-sided-bounds}. Since 
$u \mapsto \tfrac{u}{2-u}$ is increasing on $(0,2)$ and $I_d<\sqrt2$ by item~1 of Lemma~\ref{lem:C.d.I.d.inequalities},
\[
\frac{I_d}{2-I_d}<\frac{\sqrt2}{2-\sqrt2}=1+\sqrt2,
\]
which proves the dimension-uniform upper bound.

For the first lower term, Theorem~\ref{thm:one-sided-kernel} gives
$K_{+,d}=C_d \, K_{2C_d,d}$. Hence both the squared discrepancy and its
squared initial value are $C_d$ times the corresponding quantities for the
distance-kernel with $\alpha_d=2C_d$. Their normalized information
complexities coincide. Substituting $\alpha_d=2C_d$ into
the first lower term of Theorem~\ref{thm1} gives
\[
N^{\cp,+}(\varepsilon,d)
\geq\frac1{10}\left(
\frac1{\sqrt d}\,\frac1{\varepsilon^2}\,
\frac{C_dI_d}{2C_d-C_dI_d}
\right)^{\tfrac{d}{d+1}},
\]
which is the first lower bound in \eqref{eq:N-one-sided-bounds}.

It remains to prove the dimension-sharp assertion.  By
\eqref{eq:one-sided-stolarsky} and
\eqref{eq:one-sided-initial-error}, the normalized criterion $L_2^{\cp,+}(P_{N,d}) \le \varepsilon L_2^{\cp,+}(\emptyset)$ is
equivalent to
\begin{equation}\label{eq:one-sided-distance-threshold}
\frac1{N^2}\sum_{m,n=1}^N\|\bm{x}_m-\bm{x}_n\|
\geq I_d-\varepsilon^2(2-I_d).
\end{equation}
The expression on the right is positive.  Indeed, by \eqref{eq:I.d.value}, $I_1=\tfrac{4}{\pi}>1$, while
item~1 of Lemma~\ref{lem:C.d.I.d.inequalities} gives $I_d>1$ for $d\geq2$;
therefore
$I_d-\varepsilon^2(2-I_d) > 1 - \varepsilon^2(2-1) = 1 - \varepsilon^2 >0$.
Moreover,
\[
0<I_d-\varepsilon^2(2-I_d)<I_d<\sqrt2.
\]
Consequently, the denominator in \eqref{eq:R-one-sided} lies in $(0,1)$,
so $R_d(\varepsilon)>1$ is well defined.
Combining \eqref{eq:one-sided-distance-threshold} with
Lemma~\ref{lem:mean-chordal-simplex} and squaring gives
\[
2\left(1-\frac1N\right)
\geq\left(I_d-\varepsilon^2(2-I_d)\right)^2.
\]
Thus $N\geq R_d(\varepsilon)$, and the integrality of $N$ proves the second
lower bound in \eqref{eq:N-one-sided-bounds}.

Now suppose that $\lceil R_d(\varepsilon)\rceil\leq d+2$ and put
$N=\lceil R_d(\varepsilon)\rceil$. By Lemma~\ref{lem:mean-chordal-simplex}, a regular centered $(N-1)$-simplex exists in
$\mathbb{S}^d$ and attains equality in
\eqref{eq:mean-chordal-simplex}.  The definition of $R_d(\varepsilon)$ then
shows that this simplex satisfies
\eqref{eq:one-sided-distance-threshold}.  It therefore meets the normalized
error criterion, proving \eqref{eq:N-one-sided-exact}.

The uniform upper bound in \eqref{eq:N-one-sided-bounds} gives $p^*\leq2$.
Conversely, if a uniform strong-tractability estimate holds with an exponent
$p<2$, choose and fix $d$ so large that ${\frac{2d}{d+1}>p}$.  For this fixed
dimension, the first lower bound in \eqref{eq:N-one-sided-bounds} is a
positive constant times $\varepsilon^{-\frac{2d}{d+1}}$, contradicting an
$\mathcal O(\varepsilon^{-p})$ bound as $\varepsilon\downarrow0$.
Therefore $p^*\geq2$, and hence $p^*=2$.
\end{proof}

\begin{corollary}\label{cor:one-sided-eventual}
For every fixed $\varepsilon\in(0,1)$,
\begin{equation}\label{eq:N-one-sided-eventual}
N^{\cp,+}(\varepsilon,d)=M_+(\varepsilon)
\qquad\text{for all sufficiently large $d$,}
\end{equation}
where
\begin{equation}\label{eq:M-one-sided}
M_+(\varepsilon)
:=
\left\lceil
\frac1{2(\sqrt2-1)\varepsilon^2-(3-2\sqrt2)\varepsilon^4}
\right\rceil.
\end{equation}
Furthermore,
\begin{equation}\label{eq:M-one-sided-asymptotic}
M_+(\varepsilon)
=
\left\lceil
\frac{1+\sqrt2}{2} \, \frac{1}{\varepsilon^{2}} + \frac{1}{4} + \mathcal{O}(\varepsilon^{2})
\right\rceil
\qquad
\text{as $\varepsilon\downarrow0$}.
\end{equation}
Thus the dependence on $d$ is determined exactly in high dimension.
\end{corollary}

\begin{proof}
Since $I_d\to\sqrt2$ as $d \to \infty$ and $I_d<\sqrt2$ for all $d \in \mathbb{N}$, we have
\begin{align*}
R_d(\varepsilon)&\rightarrow
\left(1-\frac12
\left[\sqrt2-\varepsilon^2(2-\sqrt2)\right]^2\right)^{-1}=\frac1{2(\sqrt2-1)\varepsilon^2
-(3-2\sqrt2)\varepsilon^4}
=:R_\infty(\varepsilon).
\end{align*}
The positive quantity
$I_d-\varepsilon^2(2-I_d)=(1+\varepsilon^2)I_d-2\varepsilon^2$ is strictly
increasing in $I_d$.  Therefore $I_d<\sqrt2$ also gives
$R_d(\varepsilon)<R_\infty(\varepsilon)$.  Hence
$\lceil R_d(\varepsilon)\rceil\leq M_+(\varepsilon)$, so the hypothesis of
\eqref{eq:N-one-sided-exact} holds whenever $d+2\geq M_+(\varepsilon)$.
Moreover, if $M=M_+(\varepsilon)=\lceil R_\infty(\varepsilon)\rceil$, then
convergence to $R_\infty(\varepsilon)$ from below implies
$M-1<R_d(\varepsilon)\leq M$ for all sufficiently large $d$.  This remains
true when $R_\infty(\varepsilon)$ is an integer and proves
$\lceil R_d(\varepsilon)\rceil=M_+(\varepsilon)$ eventually.  Equation
\eqref{eq:N-one-sided-exact} now proves
\eqref{eq:N-one-sided-eventual}.  Finally, expanding
\eqref{eq:M-one-sided} as $\varepsilon\downarrow0$ gives
\eqref{eq:M-one-sided-asymptotic}.
\end{proof}

The difference from the classical discrepancy becomes apparent after normalization.
The classical kernel $K_{1,d}(\bm{x},\bm{y})=1-C_d\|\bm{x}-\bm{y}\|$ converges to the
constant kernel $1$, but the normalized one-sided kernel remains nonconstant:
\[
\frac{1}{C_d} K_{+,d}(\bm{x},\bm{y})
=1-\frac12\|\bm{x}-\bm{y}\|, 
\qquad \bm{x},\bm{y} \in \mathbb{S}^d.
\]
Moreover, \eqref{eq:one-sided-random-error} and
\eqref{eq:one-sided-initial-error} imply
\[
\frac{\mathbb{E}\big[(L_2^{\cp,+}(P_{N,d}))^2\big]}
{(L_2^{\cp,+}(\emptyset))^2}
=\frac1N\frac{I_d}{2-I_d}\rightarrow\frac{1+\sqrt2}{N}\qquad \mbox{as } d \rightarrow \infty.
\]
Thus, both the root-mean-square random-point error and the initial error decay like
$d^{-1/4}$, while their normalized ratio remains of constant order.  The
sharp simplex argument above is stronger: it
shows that, for every fixed $\varepsilon$, the optimal number of points is
eventually the explicit integer $M_+(\varepsilon)$ and hence has no residual
dependence on $d$.

\section{The hemisphere discrepancy and geodesic distance}\label{sec_hemisphere}

In \cite{BD,BDM18} the authors consider the spherical cap discrepancy
with respect to hemispheres, which is related to geodesic distance.  This
problem is distinct from the distance family \eqref{Kalpha}: that family is
only defined for $\alpha_d>C_dI_d$ and consists of kernels that are affine functions of the chordal distance $\|\bm{x}-\bm{y}\|$, whereas the hemisphere kernel below is an affine function of the geodesic distance. The associated reproducing kernel is (see \cite{BDM18})
\begin{equation*}
K_{{\rm hem},d}(\bm{x},\bm{y})
=2\int_{\mathbb{S}^d}
\bm{1}_{C(\bm{z};0)}(\bm{x})\bm{1}_{C(\bm{z};0)}(\bm{y})
\,\mathrm{d}\sigma_d(\bm{z})
=1-d_g(\bm{x},\bm{y}),
\end{equation*}
where
$d_g(\bm{x},\bm{y}):=\pi^{-1}\arccos\langle\bm{x},\bm{y}\rangle$ is the
normalized geodesic distance.

Define the squared hemisphere discrepancy by
\begin{equation*}
(L_2^{\cp,{\rm hem}}(P_{N,d}))^2
:=2\int_{\mathbb{S}^d}\left(
\frac1N\sum_{n=1}^N\bm{1}_{C(\bm{z};0)}(\bm{x}_n)-\frac12
\right)^2\,\mathrm{d}\sigma_d(\bm{z}).
\end{equation*}
The corresponding Stolarsky invariance principle is
\cite[Theorem~3.1]{BDM18}
\begin{align*}
&\frac{1}{N^2} \sum_{m,n=1}^N K_{{\rm hem},d}(\bm{x}_m, \bm{x}_n)
- \int_{\mathbb{S}^d} \int_{\mathbb{S}^d} K_{{\rm hem},d}(\bm{x}, \bm{y})
\,\mathrm{d} \sigma_d(\bm{x}) \,\mathrm{d} \sigma_d(\bm{y}) \\
&\quad=\int_{\mathbb{S}^d} \int_{\mathbb{S}^d} d_g(\bm{x}, \bm{y})
\,\mathrm{d} \sigma_d(\bm{x}) \,\mathrm{d} \sigma_d(\bm{y})
- \frac{1}{N^2} \sum_{m,n=1}^N d_g(\bm{x}_m, \bm{x}_n)=(L_2^{\cp,{\rm hem}}(P_{N,d}))^2.
\end{align*}
Here $\sigma_d(C(\bm{x};0))=1/2$, independently of $d$, and by \cite{BDM18},
\[
\int_{\mathbb{S}^d}\int_{\mathbb{S}^d}d_g(\bm{x},\bm{y})
\,\mathrm{d}\sigma_d(\bm{x})\,\mathrm{d}\sigma_d(\bm{y})=\frac12.
\]
Consequently, the initial discrepancy is $L_2^{\cp,{\rm hem}}(\emptyset)=1/\sqrt{2}$. A one-point rule also has squared discrepancy $1/2$, because its local error has absolute value $1/2$ almost everywhere; hence it cannot satisfy the normalized criterion $L_2^{\cp,{\rm hem}}(P_{N,d}) \le \varepsilon L_2^{\cp,{\rm hem}}(\emptyset)$ when $\varepsilon \in (0,1)$.  On the other hand, an antipodal pair has zero hemisphere discrepancy by \cite[Theorem~3.2]{BDM18}.  Therefore the inverse of the hemisphere discrepancy satisfies $N^{\cp,{\rm hem}}(\varepsilon,d)=2$ for every $\varepsilon\in(0,1)$ and $d\in\mathbb{N}$.

\begin{appendix}
\section[Proofs for the auxiliary estimates]{Proof of
Lemma~\ref{lem:C.d.I.d.inequalities} and of Eq.~\eqref{asym:Cxt}}\label{sec_dist_int} 

\begin{proof}[Proof of Lemma~\ref{lem:C.d.I.d.inequalities}] 

The quantities $I_d$ and $C_d$ are sequences indexed by
$d\in\mathbb{N}$.  We repeatedly use the strict log-convexity of the Gamma
function,
\[
\Gamma\left(\tfrac{u+v}{2}\right)<\sqrt{\Gamma(u)\Gamma(v)}, 
\qquad u,v>0,\ u\ne v.
\]
For $d=1$, the two lower bounds in items 1 and 2 of
Lemma~\ref{lem:C.d.I.d.inequalities} reduce to $0<I_1$ and $0<C_1$ and are
immediate.  We may therefore assume $d\geq2$ in their derivations below.

\paragraph{Lower bound for $C_d$:} 
Log-convexity of the Gamma function gives
\begin{equation}
\label{eq:Gamma.d.2.bound}
\Gamma(\tfrac{d}{2}) 
= \Gamma(\tfrac{1}{2}( \tfrac{d-1}{2} + (\tfrac{d-1}{2} + 1))) 
< \sqrt{\Gamma(\tfrac{d-1}{2})} \, \sqrt{\Gamma(\tfrac{d-1}{2} + 1 )} = \sqrt{\tfrac{d-1}{2}} \, \Gamma(\tfrac{d-1}{2}).
\end{equation}
Hence
\begin{equation*}
C_d 
= 
\frac{1}{d} \, \frac{\Gamma( \tfrac{d+1}{2} )}{\sqrt{\pi} \, \Gamma( \tfrac{d}{2} )} 
> 
\frac{1}{d} \, \frac{\tfrac{d-1}{2} \ \Gamma(\tfrac{d-1}{2})}{\sqrt{\pi} \, \sqrt{\tfrac{d-1}{2}} \, \Gamma( \tfrac{d-1}{2} )}
\end{equation*}
and the lower bound for $C_d$ follows. 

\paragraph{Upper bound for $C_d$:} Using \eqref{eq:Gamma.d.2.bound} with $d$ replaced by $d+1$ we obtain 
\begin{equation*}
C_d 
= 
\frac{1}{d} \frac{\Gamma(\tfrac{d+1}{2})}{\sqrt{\pi} \, \Gamma(\tfrac{d}{2})} 
< 
\frac{1}{d} \, \frac{\sqrt{\tfrac{d}{2}} \, \Gamma(\tfrac{d}{2})}{\sqrt{\pi} \, \Gamma(\tfrac{d}{2})} = \frac{1}{\sqrt{2\pi d}}.
\end{equation*}

\paragraph{Lower bound for $I_d$:} 
We estimate the denominator of \eqref{eq:I.d.value} from above using the log-convexity inequality and the duplication formula \cite[5.5.5]{NIST} for the Gamma function 
\begin{equation*}
\begin{split}
\Gamma( d + \tfrac{1}{2} ) 
&= \Gamma(\tfrac{1}{2}( d + ( d + 1 ) ) ) < \sqrt{\Gamma( d ) } \, \sqrt{\Gamma( d + 1 )} = \sqrt{d} \, \Gamma( d ) \\
&= \sqrt{d} \, 2^{d-1} \, \frac{1}{\sqrt{\pi}} \, \Gamma(\tfrac{d}{2}) \, \Gamma(\tfrac{d+1}{2})
\end{split}
\end{equation*}
and get
\begin{equation*}
I_d 
= 
2^d \, \frac{\Gamma( \tfrac{d + 1}{2} ) \, \Gamma( \tfrac{d + 1}{2} )}{\sqrt{\pi} \, \Gamma( d + \tfrac{1}{2} )}
> 2 \, \frac{\Gamma( \tfrac{d + 1}{2})}{\sqrt{d} \, \Gamma( \tfrac{d}{2} )}. 
\end{equation*}
Applying \eqref{eq:Gamma.d.2.bound}, we arrive at
\begin{equation*}
I_d > 2 \, \frac{\tfrac{d - 1}{2} \ \Gamma( \tfrac{d - 1}{2} )}{\sqrt{d} \, \sqrt{\tfrac{d-1}{2}} \, \Gamma( \tfrac{d-1}{2} )}
\end{equation*}
from which the lower bound follows. 

\paragraph{Upper bound for $I_d$:} The duplication formula for the Gamma function applied to the denominator of $I_d$ gives
\begin{equation} \label{eq:Id.b}
I_d = \sqrt{2} \, \frac{\Gamma(\tfrac{d}{2} + \tfrac{1}{2} ) \, \Gamma( \tfrac{d}{2} + \tfrac{1}{2} )}{\Gamma(\tfrac{d}{2} + \tfrac{1}{4} ) \, \Gamma( \tfrac{d}{2} + \tfrac{3}{4} )}.
\end{equation}
We shall show that the function
\begin{equation*}
f( \alpha ) := \Gamma( \tfrac{d}{2} + \alpha ) \, \Gamma( \tfrac{d}{2} + 1 - \alpha )
\end{equation*} 
has a unique minimum at $\alpha = \tfrac{1}{2}$ in $[0,1]$ for each $d \in \mathbb{N}$. This shows the upper bound $I_d < \sqrt{2}$. The first derivative of $f$ can be expressed in terms of the digamma function $\psi( s ) := \Gamma^\prime(s) / \Gamma( s )$ for $s \neq 0, -1, -2, \dots$, i.e.
\begin{equation*}
\begin{split}
f^\prime( \alpha ) 
&= f( \alpha ) \left( \psi(\tfrac{d}{2} + \alpha ) - \psi( \tfrac{d}{2} + 1 - \alpha ) \right) = f( \alpha ) \, \sum_{n=0}^\infty \frac{2\alpha-1}{(n + \tfrac{d}{2} + 1 - \alpha ) \, ( n + \tfrac{d}{2} + \alpha )}, 
\end{split}
\end{equation*}
where we used the series expansion \cite[5.7.6]{NIST}
\begin{equation*}
\psi( s ) = - \gamma + \sum_{n=0}^\infty \left( \frac{1}{n+1} - \frac{1}{n + s} \right), \qquad s \neq 0, -1, -2, \dots
\end{equation*}
and $\gamma$ denotes the Euler-Mascheroni constant $\gamma := \lim_{n\to \infty} \left( \sum_{k=1}^n \frac{1}{k} - \log n \right)$. 

Clearly, $f^\prime(1/2)=0$; moreover, $f$ is strictly decreasing on
$[0,1/2]$ and strictly increasing on $[1/2,1]$.  Hence $f$ has a unique
minimum at $\alpha=1/2$.

Asymptotic expansions for Gamma-function ratios
(cf. \cite[5.11.13, 5.11.15]{NIST}) applied to \eqref{eq:Id.b} yield
\begin{align*}
I_d
&=\sqrt2\left(\frac d2\right)^{1/4}
\left(1-\frac1{16d}+\mathcal O\Big( \frac{1}{d^2} \Big) \right)
\left(\frac d2\right)^{-1/4}
\left(1-\frac1{16d}+\mathcal O\Big( \frac{1}{d^2} \Big) \right)\\
&=\sqrt2\left(1-\frac1{8d}+\mathcal O\Big( \frac{1}{d^2} \Big)\right) \qquad \mbox{as } d\to\infty.
\end{align*}
The result follows. 

\paragraph{Lower and upper bounds for $C_dI_d$.} These bounds follow
by multiplying the corresponding bounds for $C_d$ and $I_d$.

\paragraph{Lower and upper bounds for $C_dI_d/(1-C_dI_d)$.} These
bounds follow from the bounds for $C_dI_d$ and simplification.
\end{proof}

\begin{proof}[Proof of \eqref{asym:Cxt}]
Let $t \in (0,1)$ be fixed. By the Funk-Hecke formula \cite{M06} we get
\begin{equation*}
\sigma_d(C(\bm{x};t)) = \frac{\omega_{d-1}}{\omega_d} \int_t^1 (1-\tau^2)^{\frac{d}{2}-1} \mathrm{d} \tau.   
\end{equation*}
Here, $\omega_d$ denotes the surface area of the sphere $\mathbb{S}^d$. In particular
\begin{equation*}
\frac{\omega_{d-1}}{\omega_d} = \frac{\Gamma(\tfrac{d+1}{2})}{\sqrt{\pi}\  \Gamma(\frac{d}{2})} = d \, C_d. 
\end{equation*}
The change of variable $(1-t^2)u=1-\tau^2$, with
$t\leq\tau\leq1$ and $0\leq u\leq1$, satisfies
$(1-t^2)\,\mathrm{d}u=-2\tau\,\mathrm{d}\tau$ and yields
\begin{equation*}
\sigma_d(C(\bm{x};t)) = \frac{d}{2} \, C_d \left( 1 - t^2 \right)^{\frac{d}{2}} \int_0^1 u^{\frac{d}{2}-1} \left( 1 - ( 1 - t^2 ) \, u \right)^{-\frac{1}{2}} \mathrm{d} u. 
\end{equation*}
Integration by parts gives
\begin{equation*}
\sigma_d(C(\bm{x};t)) = C_d \, \frac{( 1 - t^2 )^{\frac{d}{2}}}{t} - \frac{1}{2} \, C_d \left( 1 - t^2 \right)^{\frac{d}{2}} \left( 1 - t^2 \right) \int_0^1 u^{\frac{d}{2}} \left( 1 - ( 1 - t^2 ) \, u \right)^{-\frac{3}{2}} \mathrm{d} u.
\end{equation*}
Since
\begin{equation*}
0 \leq \int_0^1 u^{\frac{d}{2}} \left( 1 - ( 1 - t^2 ) \, u \right)^{-\frac{3}{2}} \mathrm{d} u \leq \frac{1}{t^3} \int_0^1 u^{\frac{d}{2}} \, \mathrm{d} u = \frac{1}{\tfrac{d}{2}+1} \, \frac{1}{t^3} 
\end{equation*}
and using Lemma~\ref{lem:C.d.I.d.inequalities}, we arrive at the two-sided estimates
\begin{equation}
\label{eq:sigma.d.estimate}
\frac{\sqrt{1 - \tfrac{1}{d}}}{\sqrt{2\pi d}} \, \frac{( 1 - t^2 )^{\frac{d}{2}}}{t} \left( 1  - \frac{1}{d+2} \, \frac{1}{t^2} \right) 
<
\sigma_d(C(\bm{x};t))
<
\frac{1}{\sqrt{2\pi d}} \, \frac{( 1 - t^2 )^{\frac{d}{2}}}{t}
\end{equation}
or
\begin{equation*}
\sigma_d(C(\bm{x};t)) = C_d \, \frac{( 1 - t^2 )^{\frac{d}{2}}}{t} \left( 1 + \mathcal{O}\left( \frac{1}{d+2} \, \frac{1}{t^2} \right) \right) \qquad \text{as $d \to \infty$,}
\end{equation*}
where the absolute constant implicit in the relative error term may be
taken to be $1$.  The result \eqref{asym:Cxt} follows from
\begin{equation*}
\left( 1 - t^2 \right)^{\frac{d}{2}} 
= 
\exp\Big( \frac{d}{2} \log( 1 - t^2 ) \Big) 
=
\exp\left( - \frac{d}{2} \left( t^2 + \frac{t^4}{2} + \cdots \right) \right)
\end{equation*}
using the convergent power-series expansion of $\log(1-u)$ at
$u=t^2\in(0,1)$.
\end{proof}

\section[Proof of the discrepancy lower bound]{Proof of
Eq.~\eqref{L2_lower_bound}}
\label{app:lower.bound}

\begin{proof}[Proof of Eq.~\eqref{L2_lower_bound}]
We apply the method of \cite{BB25}.  By \cite[Eq.~(5)]{BB25},
\begin{equation*}
\frac{1}{\sqrt{2}} \| \bm{x} - \bm{y} \| 
=
1 + \sum_{\ell=1}^\infty (-1)^{\ell} \binom{\tfrac{1}{2}}{\ell} \langle \bm{x}, \bm{y} \rangle^\ell.
\end{equation*}
This series converges absolutely for every
$\langle\bm{x},\bm{y}\rangle\in[-1,1]$.  The coefficients are negative for
$\ell\in\mathbb{N}$.  In terms of
the Pochhammer symbols $(a)_0:=1$ and
$(a)_{n+1}:=(a)_n(n+a)$ for $n\in\mathbb{N}_0$, they are
\begin{equation*}
(-1)^{\ell} \binom{\tfrac{1}{2}}{\ell} = \frac{(-\frac{1}{2})_{\ell}}{\ell!} = \frac{1}{\Gamma( - \frac{1}{2} )} \, \frac{\Gamma( \ell - \frac{1}{2} )}{\Gamma( \ell + 1 )}.
\end{equation*}
Substitution of the distance expansion into \eqref{eq:Stolarsky.gen.A}, using \eqref{Kalpha}, gives
\begin{eqnarray}
\lefteqn{(L_2^{\cp,\alpha_d}(P_{N,d}))^2} \notag \\
& = & \frac{1}{N^2} \sum_{m,n=1}^N K_{\alpha_d,d}(\bm{x}_m, \bm{x}_n) - \int_{\mathbb{S}^d} \int_{\mathbb{S}^d} K_{\alpha_d,d}(\bm{x}, \bm{y}) \,\mathrm{d} \sigma_d(\bm{x}) \,\mathrm{d} \sigma_d(\bm{y}) \notag \\
& = &
\frac{\sqrt{2} C_d}{\alpha_d} \sum_{\ell=1}^\infty \frac{-(-\frac{1}{2})_{\ell}}{\ell!} \left( \frac{1}{N^2} \sum_{m,n=1}^N \langle \bm{x}_m, \bm{x}_n \rangle^\ell - \int_{\mathbb{S}^d} \int_{\mathbb{S}^d} \langle \bm{x}, \bm{y} \rangle^\ell \,\mathrm{d} \sigma_d(\bm{x}) \,\mathrm{d} \sigma_d(\bm{y}) \right). \label{eq:L2.cap.a.d.A}
\end{eqnarray} 
The parenthetical expressions are all nonnegative by
\cite[Lemma~1]{BB25}. Thus the series consists of nonnegative terms. By
removing all odd terms and all even terms with $\ell<2M$, where
$M\in\mathbb{N}$ is chosen below, we get
\begin{eqnarray*}
\lefteqn{(L_2^{\cp,\alpha_d}(P_{N,d}))^2}\\
& \geq &
\frac{\sqrt{2} C_d}{\alpha_d} \sum_{r=M}^\infty \frac{-(-\frac{1}{2})_{2r}}{(2r)!} \left( \frac{1}{N^2} \sum_{m,n=1}^N \langle \bm{x}_m, \bm{x}_n \rangle^{2r} - \int_{\mathbb{S}^d} \int_{\mathbb{S}^d} \langle \bm{x}, \bm{y} \rangle^{2r} \,\mathrm{d} \sigma_d(\bm{x}) \,\mathrm{d} \sigma_d(\bm{y}) \right).
\end{eqnarray*}
By \cite[Lemma~2]{BB25} 
\begin{equation*}
\int_{\mathbb{S}^d} \int_{\mathbb{S}^d} \langle \bm{x}, \bm{y} \rangle^{2r} \,\mathrm{d} \sigma_d(\bm{x}) \,\mathrm{d} \sigma_d(\bm{y})  
=  
\frac{(\frac{1}{2})_{r}}{(\frac{d+1}{2})_{r}} 
=
\frac{\Gamma(\tfrac{d+1}{2})}{\Gamma(\tfrac{1}{2})} \, \frac{\Gamma(r+\tfrac{1}{2})}{\Gamma(r+\tfrac{d+1}{2})}
\sim
\frac{\Gamma( \frac{d+1}{2} )}{\Gamma( \frac{1}{2} )} \, \frac{1}{r^{\frac{d}{2}}} \quad \text{as $r \to \infty$.}    
\end{equation*}
Furthermore, these integrals are strictly decreasing in $r$ (as can be seen from the ratio of the Pochhammer symbols), and the
displayed asymptotic expression is a strict upper bound for every
$r\in\mathbb{N}$ (cf. Proof of \cite[Lemma~2]{BB25}).  Using this fact and retaining only the diagonal terms in
the double sum, we obtain
\begin{equation*}
\frac{1}{N^2} \sum_{m,n=1}^N \langle \bm{x}_m, \bm{x}_n \rangle^{2r} - \int_{\mathbb{S}^d} \int_{\mathbb{S}^d} \langle \bm{x}, \bm{y} \rangle^{2r} \,\mathrm{d} \sigma_d(\bm{x}) \,\mathrm{d} \sigma_d(\bm{y}) 
\geq 
\frac{1}{N} - \frac{(\frac{1}{2})_{M}}{(\frac{d+1}{2})_{M}} 
\geq 
\frac{1}{N} - \frac{\Gamma( \frac{d+1}{2} )}{\Gamma( \frac{1}{2} )} \, \frac{1}{M^{\tfrac{d}{2}}}.
\end{equation*}
By \cite[Eq.~(25)]{BB25} we have
\begin{equation*}
\sum_{r=M}^\infty \frac{-(-\frac{1}{2})_{2r}}{(2r)!} 
\geq
\frac{1}{2 \sqrt{2 \pi}} \, \frac{1}{M^{\tfrac{1}{2}}} \qquad \mbox{for all } M \in \mathbb{N}.
\end{equation*}
The right-hand side is the leading term of the large-$M$ asymptotics.
Consequently,
\begin{equation}
\label{eq:L.2.cap.lower.bound}
(L_2^{\cp,\alpha_d}(P_{N,d}))^2  \geq 
\frac{\sqrt{2} C_d}{\alpha_d} \frac{1}{2 \sqrt{2 \pi}}  \left( \frac{1}{N} - \frac{\Gamma( \frac{d+1}{2} )}{\Gamma( \frac{1}{2} )} \, \frac{1}{M^{\tfrac{d}{2}}} \right) \frac{1}{M^{\tfrac{1}{2}}} = \frac{\sqrt{2} C_d}{\alpha_d} \frac{1}{2 \sqrt{2 \pi}} \, \frac{F( \frac{M^{d/2}}{N} )}{N^{1+\frac{1}{d}}} 
\end{equation}
provided $M$ is such that the parenthetical expression is positive. The function 
\begin{equation*}
F(x) := \frac{x - \beta}{x^{1+\gamma}}, \qquad x \geq \beta := \frac{\Gamma( \frac{d+1}{2} )}{\Gamma( \frac{1}{2} )}, \qquad \gamma := \frac{1}{d}
\end{equation*}
has a unique maximum at $x^* = \beta ( 1 + \frac{1}{\gamma} ) > \beta$ with value 
\begin{equation} 
\label{eq:F.x.star}
F(x^*) = \frac{1}{( 1 + \gamma ) \big( \beta ( 1 + \frac{1}{\gamma} ) \big)^{\gamma}}=\frac{d}{(d+1)^{1+\frac{1}{d}}} \left(\frac{\Gamma(\frac{1}{2})}{\Gamma(\frac{d+1}{2})}\right)^{\frac{1}{d}}.
\end{equation}
It remains to choose an integer $M$.  Write
$x_d^*:=x^*=\beta(d+1)$, set
\[
a
:=
(Nx_d^*)^{\tfrac{2}{d}},
\qquad M:=\lceil a\rceil,
\qquad r:=\frac{M}{a},
\qquad y:=r^{\tfrac{d}{2}}.
\]
We first prove the uniform estimate $a\geq\tfrac{4}{\pi}$. For $d=1$ this bound follows
from $x_1^*=\tfrac{2}{\sqrt{\pi}}$ and $N\in \mathbb{N}$.  For $d\geq2$ we have 
\[
x_d^*=(d+1)\frac{\Gamma(\tfrac{d+1}{2})}{\sqrt\pi}.
\]
The inequalities $x_2^*=\tfrac{3 \Gamma(3/2)}{\sqrt{\pi}} = \frac{3}{2}$ and $x_3^*=\tfrac{4}{\sqrt{\pi}} \geq (\tfrac{3}{2})^{\tfrac{3}{2}}$, together with the recurrence relation
\[
x_{d+2}^*=\frac{d+3}{2}x_d^*\geq\frac32x_d^*,
\]
show by induction on each parity that
$x_d^*\geq(\tfrac{3}{2})^{\tfrac{d}{2}}$.  Hence $a\geq\tfrac{3}{2}>\tfrac{4}{\pi}$ when $d\geq2$.
Consequently,
\[
1\leq r<1+\frac1a\leq1+\frac\pi4.
\]
Moreover, $y \geq 1$ implies $\tfrac{M^{d/2}}{N}=y \, x_d^* \geq x_d^* = \beta ( d + 1 ) > \beta$ and thus $F(\tfrac{M^{d/2}}{N}) > 0$, so the parenthetical expression in \eqref{eq:L.2.cap.lower.bound} is positive. 
A direct calculation using
$x_d^*=\beta(d+1)$ gives
\begin{align*}
\frac{F(yx_d^*)}{F(x_d^*)}
&=\frac{(d+1)y-1}{d\,y^{1+1/d}}
\geq y^{-\tfrac{1}{d}}=r^{-\tfrac{1}{2}}
\geq\sqrt{\frac4{4+\pi}} =: \rho_0.
\end{align*}
Substitution of this integer $M$ into the preceding discrepancy estimate \eqref{eq:L.2.cap.lower.bound},
followed by $\sqrt2>I_d$, proves \eqref{L2_lower_bound}.
\end{proof}

\end{appendix}

\paragraph{Acknowledgement} 

We thank Erich Novak for suggesting that we investigate the one-sided spherical cap $L_2$ discrepancy and an anonymous referee for valuable remarks and inspiring questions.

The work of J. D. was supported in part by a grant of access to OpenAI models through the ChatGPT for Academic Researchers program.

\paragraph{Declaration of generative AI use}

The authors used ChatGPT 6 Astra and 6 Sol for the revision of the article. A final review of the paper has been done using Claude Opus 5.5. The authors independently checked and verified the AI generated mathematical arguments, calculations, references, and conclusions, and they take full responsibility for the contents of the paper.

\end{document}